\documentclass[10pt, letterpaper, twoside]{article} 

\usepackage{amsmath, amsthm, amssymb,bbold,mathrsfs} 
\usepackage{graphicx}
\usepackage[toc,page]{appendix}
\usepackage{enumitem}
\usepackage{thmtools}
\usepackage{thm-restate}

\usepackage{hyperref}

\usepackage{cleveref}

\newtheorem{cor}{Corollary}[section]
\newtheorem{prop}{Proposition}[section]

\declaretheorem[name=Theorem,numberwithin=section]{thm}
\declaretheorem[name=Lemma,numberwithin=section]{lem}

\theoremstyle{remark}
\newtheorem{rem}{Remark}[section]

\theoremstyle{definition}
\newtheorem{example}{Example}[section]
\newtheorem{definition}{Definition}[section]

\numberwithin{equation}{section}
 
 \def\E{\mathbb{E}}
 \def\oE{\widetilde{\mathbb{E}}}
 \def\P{\mathcal{P}}
 \def\B{\mathcal{B}}
 \def\M{\mathcal{M}}
 \def\q{\gamma}
 \def\I{D}
 
 \def\T{{\widetilde T}} 
 \def\epi{\mathop{\text{epi}}}

 \def\mysmallskip{\vspace*{0.15cm}}
 
 \def\0{\mathbf{0}}
 \def\1{\mathbf{1}}
 
\def\old#1{{}}
 
\begin{document}

\title{On Convergence of Value Iteration for a Class \\ of Total Cost Markov Decision Processes\thanks{This research is in part supported by a grant from Alberta Innovates -- Technology Futures.}}
\author{Huizhen Yu\thanks{Department of Computing Science, University of Alberta. Email: janey.hzyu@gmail.com}}
\date{}
\maketitle

\begin{abstract}
We consider a general class of total cost Markov decision processes (MDP) in which the one-stage costs can have arbitrary signs, but the sum of the negative parts of the one-stage costs is finite for all policies and all initial states. This class, which we refer to as the General Convergence (GC for short) total cost model, contains several special classes of problems, e.g., positive costs problems, bounded negative costs problems, and discounted problems with unbounded one-stage costs. We study the convergence of value iteration for the GC model, in the Borel MDP framework with universally measurable policies.  Our main results include: (i)~convergence of value iteration when starting from certain functions above the optimal cost function; (ii)~convergence of transfinite value iteration starting from zero, in the special case where the optimal cost function is nonnegative; and (iii)~partial convergence of value iteration starting from zero, for a subset of initial states.

These results extend several previously known results about the convergence of value iteration for either positive costs problems or GC total cost problems. In particular, the first result on convergence of value iteration from above extends a theorem of van der Wal for the GC model. The second result relates to Maitra and Sudderth's analysis of transfinite value iteration for the positive costs model.
It suggests connections between the two total cost models when the optimal cost function is nonnegative, and it leads to additional results on the convergence of ordinary non-transfinite value iteration, with a suitably defined dynamic programming operator, for finite state or finite control GC problems. The third result on partial convergence of value iteration is motivated by Whittle's bridging condition for the positive costs model, and provides a novel extension of the bridging condition to the GC model, where there are no sign constraints on the costs.
\end{abstract}

\setlength{\unitlength}{1cm}
\begin{picture}(0,0)(0,0)
\put(-0.52,15.0){\fontsize{12}{14} \selectfont Technical Report}
\put(-0.52,14.4){\fontsize{12}{14} \selectfont \it Nov.\ 5, 2014 (revised Mar.\ 2015)}
\end{picture}


\clearpage
\tableofcontents

\clearpage

\section{Introduction} \label{sec1}
\markboth{\rm \S \ref{sec1}. Introduction}{\rm \S \ref{sec1}. Introduction}

In this paper we study convergence properties of value iteration for a class of Markov decision processes (MDP) under the undiscounted total cost criterion.
Specifically, we consider problems in which the one-stage costs can take both positive and negative values, but we assume that under any policy and for any starting state, the expected total sum, over the infinite horizon, of the negative parts of the one-stage cost is finite. 
Following the extensive survey on this class of MDP (Feinberg \cite{Fein02}), we shall refer to this total cost model as General Convergence total cost model (GC for short). It contains several classes of special models, in particular, the positive costs model (P), where all the one-stage costs are nonnegative, the negative costs model (N) where all the one-stage costs are nonpositive and the optimal costs are finite, as well as the bounded negative costs model \cite[Chap.\ 7.2]{puterman94} and a subset of the discounted models with unbounded costs (UD).

It is known that for the (GC) model, value iteration starting from the constant function zero need not converge to the optimal cost function, surprisingly, even for finite state and control problems (see van der Wal~\cite[Example 3.2]{vdW81}, Feinberg~\cite[Example 6.10]{Fein02}). 
This is also true for the positive costs model (P) (although convergence is ensured for (P) if each state has only a finite number of feasible controls~\cite[Prop.\ 9.18]{bs} or more generally, if certain compactness conditions (\cite{b77}, \cite[Prop.\ 9.17]{bs}) or semicontinuity and compactness model assumptions~\cite{Schal75} are satisfied).
In the (P) case, Maitra and Sudderth established the convergence of transfinite value iteration \cite{MS92}, and Whittle formulated a sufficient condition, called the bridging condition, for the convergence of value iteration starting from zero as well as from any function between zero and a multiple of the optimal cost function \cite{Whit79} (see also \cite{Har80,puterman94}). Recently, motivated by these earlier results, Yu and Bertsekas showed that for the (P) case, value iteration always converges from above, when it starts within the set of functions that are above the optimal cost function and yet bounded by a multiple of the optimal cost function~\cite{YuB-mvipi}. In this paper, we extend these results for the (P) case to the total cost problems in the (GC) model, where there are no sign constraints on the one-stage costs.
 
The main contributions of this paper are as follows:
\begin{itemize}
\item[(i)] We prove that value iteration always converges from above if it starts within a set of functions that are above the optimal cost function and yet bounded by a multiple of a certain nonnegative function. This nonnegative function is given by $J^{*+} + J^{*-}$, where $J^{*+}, J^{*-} \geq 0$ are the optimal cost functions with respect to the positive part and negative part of the one-stage cost function, respectively (see Theorem~\ref{thm-vi}).  
This result extends a theorem of van der Wal~\cite[Theorem 3.7]{vdW81} about the convergence of value iteration for (GC).
\item[(ii)] We prove that if the optimal cost function is nonnegative, then when suitably defined for (GC), transfinite value iteration starting from zero converges to the optimal cost function (see Theorem~\ref{thm-relate2P}). This result is analogous to the transfinite value iteration result of Maitra and Sudderth for positive costs problems \cite{MS92}. We also show that if the optimal cost function is nonnegative, then (a) between zero and the optimal cost function, the dynamic programming operator has no fixed points other than the optimal cost function itself (see Cor.~\ref{cor-gc1}); and (b) convergence of ordinary non-transfinite value iteration, with a suitably defined operator, can be ensured for finite control problems (see Prop.~\ref{prp-vi-pos2}), as well as for finite state problems that have finite optimal costs (see Prop.~\ref{prp-vi-pos1}).
\item[(iii)] We formulate a novel sufficient condition, in the style of Whittle's bridging condition \cite{Whit79}, for \emph{partial convergence} of value iteration in (GC) (see Theorems~\ref{thm-vi-partial},~\ref{thm-vi-partial2} and Remark~\ref{rmk-bridging}). It characterizes the convergence of value iteration starting from zero, for a \emph{subset} of states.
\end{itemize}
Besides the closely related early works mentioned above, we will discuss later, with more precise mathematical terms, some other related works in connection with our results in Sections~\ref{sec2.3} and~\ref{sec4.2-discussion}.

To present our analyses of the (GC) model for general state and control spaces, we shall consider Borel-spaces MDP in the universal measurability framework~(Shreve and Bertsekas~\cite{ShrB79}, Bertsekas and Shreve~\cite[Part II]{bs}). We shall derive various optimality and convergence properties for the (GC) model in this framework. Most of the ideas in our analyses of value iteration are, however, independent of the measurability concepts, so we will just summarize the universal measurability framework, selectively note the properties of analytic sets and universally measurable policies needed in our proofs, and refer to the original sources for some of the details. (As to the deep mathematical reason for adopting this framework, we refer the readers to the articles \cite{BFO74,Shr79} and the monograph \cite{bs}.) 
To our knowledge, the (GC) model has not been thoroughly studied in the universal measurability framework, so some of the optimality properties of the (GC) model that we present in this paper can also be of independent interest.

While our focus will be on value iteration, it is worth noting that generally, for undiscounted total cost MDP, it is not easy to find near-optimal policies or optimal ones such as optimal stationary policies (if they exist), even if the optimal cost function is available.  For positive or negative costs models as well as for the (GC) model, the existence proofs for near-optimal policies are based on constructing such a policy, using the optimal cost function in addition to some other information (cf.\ the proof of \cite[Prop.\ 9.20]{bs} and the proof of Theorem~\ref{thm-m2}(a) given in Appendix~\ref{appendix1} of this paper). Thus in theory, methods of finding a near-optimal policy are available, although they may not be computationally tractable. As to finding an optimal stationary policy, the standard policy iteration method is known to have convergence issues even when such a policy exists (see e.g., \cite[Chap.\ 7]{puterman94}); recently, in a broader context of constrained total cost MDP, a new method based on linear programming has been established by \cite{DuHP12,DuP13} for semicontinuous models under certain conditions.

This paper is organized as follows. In Section~\ref{sec2} we define formally the class of total cost problems in (GC), using the universal measurability framework, and we summarize some of the fundamental optimality properties of this model. In Section~\ref{sec3} we establish the convergence of value iteration starting within a certain set of functions above the optimal cost function. In Section~\ref{sec4} we consider value iteration starting from the constant function zero. We first show, in Section~\ref{sec4.1}, that if the optimal cost function is nonnegative, then the (GC) model exhibits some properties similar to those in the positive model (P). These include the fixed point properties of the dynamic programming operator, the convergence of transfinite value iteration, and the convergence of ordinary non-transfinite value iteration for finite state or finite control problems.
We then present in Section~\ref{sec4.2} our extension of Whittle's bridging condition. We also discuss some direct consequences of this result, including one that relates to the sufficient conditions given earlier by van Hee et al.~\cite{vanH77}, and by Hern\'{a}ndez-Lerma and Lasserre \cite{HL99}, for value iteration. In Section~\ref{sec5} we use examples to illustrate the convergence results obtained in this paper. In Section~\ref{sec6} we conclude the paper with several remarks on open questions for future research. Appendices~\ref{appendix2}-\ref{appendix3} collect some proofs or outlines of proofs that require an in-depth use of the Borel MDP and the universal measurability framework.

\section{Background} \label{sec2}
\markboth{\rm \S \ref{sec2}. Background}{\rm \S \ref{sec2}. Background}
 
\subsection{Borel-Spaces MDP and the Total Cost Model (GC)} \label{sec2.1}

We consider Borel-spaces MDP in the universal measurability framework~\cite[Part II]{bs}. 
Let $S$ and $C$ denote the \emph{state space} and the \emph{control space}, respectively. 
We assume that they are Borel spaces.\footnote{A Borel space is a Borel subset of some Polish space (a complete separable metric space).} 
We write $x$ for a state in $S$ and $u$ for a control in $C$. Associated with each state $x$ is a nonempty set $U(x) \subset C$, consisting of \emph{feasible controls} at $x$. The set-valued function $U: x \mapsto U(x)$ thus specifies the \emph{control constraint}, and its graph
$ \Gamma = \big\{ (x, u) \mid x \in S, u \in U(x) \big\}$
is assumed to be an analytic subset of $S \times C$.
\footnote{There are several equivalent definitions of analytic sets in a Polish space (see \cite[Prop.\ 7.41]{bs}, \cite[Chap.\ 13.2]{Dud02}). The empty set is analytic. A nonempty analytic set is by definition the image of some Borel set in some Polish space, under a Borel measurable mapping. In particular, if $X, Y$ are two Polish spaces and $B$ is a Borel subset of $X \times Y$, the projection of $B$ into $X$ is an analytic set in $X$.} 
A given function $g: \Gamma \to [-\infty, +\infty]$ specifies the \emph{one-stage cost} $g(x,u)$ for each state and feasible control pair $(x, u) \in \Gamma$. 
We assume that the function $g$ is lower semi-analytic.
\footnote{A function $f: D \to [-\infty, +\infty]$ is lower semi-analytic if and only if its domain $D$ is an analytic set and its level sets $\{ y \in D \mid f(y) < c\}$ are analytic for all real $c$ \cite[Def.\ 7.21]{bs}. Equivalently, $f$ is lower semi-analytic if and only if its epigraph $\{(y,c) \mid f(y) \leq c, \, y \in D, \, c \in (-\infty, + \infty) \}$ is analytic (see \cite[p.\ 186]{bs}, \cite{MS92}).}
At state $x$, if we apply a control $u \in U(x)$, a one-stage cost $g(x,u)$ is incurred, and the system then moves to the next state $x'$ according to a given stochastic kernel $q(dx' \!\mid x, u)$. More precisely, the kernel $q$ for \emph{state transition} is defined as follows. For a Borel space $Y$, we denote by $\P(Y)$ the set of probability measures on the Borel $\sigma$-algebra on $Y$. 
The state transition kernel $q(dx' \!\mid x, u)$ maps each $(x,u) \in S \times C$ to a probability measure in $\P(S)$, and furthermore, $q$ is assumed to be Borel measurable, meaning that the mapping $q: (x, u) \mapsto q(\cdot \!\mid x, u) \in \P(S)$ is Borel measurable (where the space $\P(S)$ is equipped with the weak topology).
  
For $k \geq 0$, we denote by $(x_k, u_k)$ the state and control pair at time $k$. 
We define a \emph{policy} for the above MDP to be a sequence of stochastic kernels, $\pi=(\mu_0,\mu_1, \ldots)$, where for each $k \geq 0$, $\mu_k$ is a universally measurable stochastic kernel
\footnote{For two Borel spaces $Y, Z$, a universally measurable stochastic kernel on $Y$ given $Z$ is by definition a universally measurable mapping from $Z$ to the space $\P(Y)$. Here $\P(Y)$ is equipped with the weak topology, which makes $\P(Y)$ a Borel space \cite[Cor.\ 7.25.1]{bs}, and the $\sigma$-algebra on $Z$ is the universal $\sigma$-algebra (see e.g., \cite[Def.\ 7.18]{bs} and \cite[Chap.\ 3.3]{Dud02} for the completion of measures and the universal $\sigma$-algebra).}
on $C$ given $(S \times C)^{k} \times S$ such that 
$$\mu_k(U(x_k) \mid x_0, x_1, \ldots, x_k) = 1, \qquad \forall \,  (x_0, x_1, \ldots, x_k) \in (S \times C)^{k} \times S. $$ 
We denote by $\Pi$ the set of all policies. For each initial state $x \in S$, a policy $\pi$ induces a stochastic process $(x_0, u_0, x_1, u_1, \ldots)$ on the space $(S \times C)^\infty$ (endowed with the universal $\sigma$-algebra), with $x_0=x$. The probability of this process is determined uniquely by the collection of the stochastic kernels comprising $\pi$, together with the state transition kernel $q(dx'\!\mid x, u)$ \cite[Prop.\ 7.45]{bs}. 
We denote by $\E^\pi_x$ expectation with respect to this induced probability measure on $(S \times C)^\infty$. 

We define structured families of policies in the standard way. A policy is \emph{nonrandomized} if for each $k$ and each $(x_0, u_0, \ldots, x_k)$, $\mu_k(du_k \! \mid x_0, x_1, \ldots, x_k)$ assigns probability one to a single control in $U(x_k)$. A policy is \emph{semi-Markov} (resp.\ \emph{Markov}) if for each $k$, the function $\mu_k: (x_0, u_0, \ldots, x_k) \mapsto \mu_k(\cdot \mid x_0, u_0, \ldots, x_k)$ depends only on $(x_0, x_k)$ (resp.\ $x_k$). A Markov policy of the form $\pi=(\mu, \mu, \ldots)$ is called a \emph{stationary} policy, in which case we simply write $\mu$ for that policy.

We mention that since the set $\Gamma$ of state and feasible control pairs is analytic, a nonrandomized stationary policy exists by the Jankov-von Neumann selection theorem~\cite[Prop.\ 7.49]{bs}, and therefore, the set $\Pi$ of policies is nonempty.\footnote{By comparison, if $\Gamma$ is assumed to be Borel measurable instead and only Borel measurable policies are allowed, then, without additional assumptions, the set of policies can be empty \cite{Blk-borel}.}
We refer to the monograph~\cite[Part II]{bs} for a full introduction of the universal measurability framework, including the mathematical background for the preceding definitions, the properties of lower semi-analytic functions and universally measurable policies, and measurable selection theorems upon which the framework is built.\footnote{For a quick overview of some of these subjects, we refer the readers to the articles \cite{BFO74,MS92,Shr79} and the recent paper~\cite[Section 2]{YuB-mvipi}. For a general introduction to Borel MDP with Borel measurable policies, see the book \cite{Dynkin79}; for analytic sets, see also \cite[Chap.\ 13.2]{Dud02} and \cite[Chap.\ 6.2]{MS96}.}

The class (GC) of total cost problems is defined by a finiteness condition on the sums of the negative parts of the one-stage costs under any policy. More specifically, let us write $g$ in terms of its positive part $g_+$ and negative part $g_-$ as:
$$g = g_+ - g_-, \qquad \text{where} \ \ \  g_+ = \max \{ g, 0\}, \ \ \ g_- =  \max \{-g, 0\}.$$
For any policy $\pi \in \Pi$, define
$$ J_{\pi}^+(x) = \E^\pi_x \left\{ \sum_{k=0}^\infty g_+(x_k,u_k) \right\}, \qquad J_{\pi}^-(x) = \E^\pi_x \left\{ \sum_{k=0}^\infty g_-(x_k,u_k) \right\}, \quad x \in S,  $$
where the expectations are well-defined by the monotone convergence theorem.

\begin{definition}[The (GC) total cost model]
An MDP given above is in the \emph{(GC)} class if for all policies $\pi \in \Pi$, $J_\pi^-$ is real-valued, or equivalently, if
\begin{equation} \label{eq-tmodel}
    \sup_{\pi \in \Pi} J_\pi^-(x)  < \infty, \qquad \forall \, x \in S.
\end{equation}    
\end{definition}

Note that the condition (\ref{eq-tmodel}) requires in particular $g  > - \infty$. The equivalence mentioned in the preceding definition can be verified directly. As a result of the condition (\ref{eq-tmodel}), for an MDP in the (GC) class, 
we can define the \emph{total cost of a policy} $\pi$ in several equivalent ways:
\begin{align*}
  J_\pi(x)  : = & \, \limsup_{n \to \infty} \, \E^\pi_x \left\{ \sum_{k=0}^n g(x_k,u_k) \right\} = \lim_{n \to \infty} \E^\pi_x \left\{ \sum_{k=0}^n g(x_k,u_k) \right\}  
              = \E^\pi_x \left\{ \sum_{k=0}^\infty g(x_k,u_k) \right\}  \\
              = & \, \, \E^\pi_x \left\{ \sum_{k=0}^\infty g_+(x_k,u_k) \right\}  -  \E^\pi_x \left\{ \sum_{k=0}^\infty g_-(x_k,u_k) \right\} = J_\pi^+(x) - J_\pi^-(x).
\end{align*} 
All these expressions are well-defined. In particular, under the condition (\ref{eq-tmodel}), the sums of the one-stage costs, $\sum_{k=0}^n g(x_k,u_k)$ and $\sum_{k=0}^\infty g(x_k,u_k)$, along a trajectory $(x_0, u_0, x_1, u_1, \ldots)$, are well-defined with probability one under any policy. The situation $\infty - \infty$ does not occur in these summations (except on a negligible set of trajectories) or in any of the expectations above.
The condition~(\ref{eq-tmodel}) is called the ``general convergence condition'' \cite{Fein02}, from which comes the acronym (GC) for the class of problems satisfying this condition. 

With the total cost of a policy defined as above,
the \emph{optimal cost function} $J^*$ is given by 
$$J^*(x) = \inf_{\pi \in \Pi} J_\pi(x) > - \infty, \qquad \forall \,  x \in S,$$ 
where $J^*$ does not take the value $-\infty$ by condition~(\ref{eq-tmodel}).
A policy $\pi$ is \emph{optimal for state} $x$, if $J_\pi(x) = J^*(x)$. If this holds for all $x$, we say $\pi$ is an \emph{optimal policy}. Similarly, for $\epsilon > 0$, a policy $\pi$ is \emph{$\epsilon$-optimal for state} $x$, if $J_\pi(x) \leq J^*(x) + \epsilon$; and if this holds for all $x$, we say $\pi$ is an \emph{$\epsilon$-optimal policy}.

\subsubsection*{Special Cases of the (GC) Model}

As noted earlier, the (GC) class contains several special classes of total cost problems.
\footnote{Some total cost problems are excluded from the (GC) model by the condition (\ref{eq-tmodel}). As a simple example, if certain states can form a cycle in the state space under some policy, where some of the edges of the cycle have negative costs, then the negative edges will give rise to a total amount of $-\infty$ when the system goes through the cycle infinitely many times. This violates the condition (\ref{eq-tmodel}), so even though the net cost may be nonnegative, such problems are not in the (GC) class. Shortest path or stochastic shortest path problems can be of this type. They can be analyzed as total cost problems under certain model conditions~\cite{BeT-ssp}, when the average costs of all policies are nonnegative.} 
The case $g\geq 0$ is the positive costs model (P). The case $g \leq 0$ is the negative costs model (N) with the additional condition that the optimal cost function is real-valued. The case where for each state $x$, there exists a control $u \in U(x)$ with $g(x,u) \leq 0$, is the bounded negative costs model discussed in \cite[Chap.\ 7.2]{puterman94} (where it is called positive bounded model in the reward maximization framework). 

The (GC) model also covers a set of discounted problems with unbounded one-stage costs.
Generally speaking, we can always reformulate discounted problems, with possibly transition-dependent discounting, as equivalent undiscounted total cost problems, using a standard procedure: Enlarge the state space to $S \cup \{\Delta\}$, where $\Delta$ represents an absorbing cost-free termination state, and then set the probability of transitioning to $\Delta$ from a state in $S$ to be one minus the respective discount factor. 
For simplicity, in this paper we shall only consider a subset of the discounted problems that fall into the (GC) class. The discount factor in these problems is a given number, independent of the state transitions. We will refer to this class as (UD) in the rest of the paper. 

\begin{definition}[The unbounded discounted model (UD)] \label{def-ud}
Let $\beta \in [0,1)$. An MDP given above is in the class \emph{(UD)} if, as in the definition of (GC), 
$\sup_{\pi \in \Pi} J_\pi^-(x)  < \infty$ for all $x \in S$ and the objective is to minimize $J_\pi$ over all policies $\pi \in \Pi$, but with the cost functions $J_\pi^-, J_\pi$ defined by the $\beta$-discounted total costs of $\pi$ as 
$$ J_{\pi}^-(x) = \E^\pi_x \left\{ \sum_{k=0}^\infty \beta^k g_-(x_k,u_k) \right\}, \qquad J_{\pi}(x) = \E^\pi_x \left\{ \sum_{k=0}^\infty \beta^k g(x_k,u_k) \right\}, \qquad \forall \, x \in S.$$
\end{definition}

\subsection{Some Optimality Properties of the (GC) Model}

We discuss now some general properties of the optimal cost function and optimal policies for the (GC) model in the universal measurability framework. Since we focus exclusively on the (GC) class of total cost problems, in all the theorems given in this paper, we will not state the (GC) condition~(\ref{eq-tmodel}) explicitly. 

Let $A(S)$ denote the set of lower semi-analytic functions $J: S \to [-\infty, + \infty]$, and let $\M(S)$ denote the set of universally measurable functions
$J: S \to [-\infty, + \infty]$.
\footnote{A function $J : S \to [-\infty, + \infty]$ is universally measurable if it is measurable with respect to the universal $\sigma$-algebra on $S$. A lower semi-analytic function is universally measurable, so $A(S) \subset \M(S)$ \cite[Chap.\ 7.7]{bs}.} 
For the (GC) model, the dynamic programming operator $T: A(S) \to A(S)$ is defined by 
\begin{equation} \label{eq-T}
    T(J)(x) : = \inf_{u \in U(x)} \left\{ g(x,u) + \int_S J(x') \, q(dx' \! \mid x, u) \right\},  \qquad x \in S, \ \forall \, J \in A(S).
\end{equation}    
Associated with a stationary policy $\mu$ is the dynamic programming operator $T_\mu:\M(S) \to \M(S)$ defined by
$$ T_\mu(J)(x) : = \int_{C}  \left\{ g(x,u) + \int_S J(x') \, q(dx' \! \mid x, u) \right\} \, \mu(du \!\mid x), \qquad x \in S, \  \forall \, J \in \M(S).$$
In the above, the integral $\int_S J(x') \, q(dx' \! \mid x, u)$ appearing in $T$ and $T_\mu$ is defined with respect to the completion of the Borel probability measure $q(\cdot \! \mid x, u)$ (see \cite[p.\ 173]{bs} and see \cite[Chap.\ 3.3]{Dud02} for the completion of measures), and the integral $\int_C \{ \cdots\} \, \mu(du \!\mid x)$ in $T_\mu$ is defined likewise. 

We clarify another detail in the definitions of $T, T_\mu$ given above, which concerns operations involving extended real-valued functions. In general, the integral $\int f d\rho$ for a Borel probability measure $\rho$ and an extended real-valued, universally measurable function $f$ is defined as $\int f d\rho = \int f_+ d\rho - \int f_- d\rho$, where $f_+, f_-$ are the positive and negative parts of $f$, respectively. We use the convention $\infty - \infty = - \infty + \infty = \infty$ in this definition; we note, however, that such arithmetic operations will not be encountered for the (GC) model in this paper. In particular, under the (GC) condition~(\ref{eq-tmodel}), it can be verified that we do not encounter $\infty - \infty$ or $-\infty + \infty$ when computing $T(J^*)$, $T_\mu(J^*)$, or more generally, when computing $T(J)$ for any $J \in A(S)$ satisfying $J(x) \geq - \sup_{\pi \in \Pi} J_\pi^-(x)$, $x \in S$. At this point, however, there is still no need to introduce a general set of such functions $J$; such a set will be formally introduced later in Section~\ref{sec4.2}.

We can state now some important properties of the (GC) model. For any policy $\pi$, $J_\pi \in \M(S)$ \cite[Chap.\ 9.1]{bs}, and for any stationary policy $\mu$, $J_\mu$ satisfies the dynamic programming equation $J_\mu = T_\mu(J_\mu)$
(see Lemma~\ref{lma-DPpolicy} in Appendix~\ref{appsec2}). Furthermore, it can be shown that the (GC) model has the following optimality properties:

\begin{restatable}[Properties of the optimal cost function]{thm}{thmoptone} \label{thm-m1} 
\hfill
\begin{itemize}
\item[(a)] The optimal cost function $J^*$ is lower semi-analytic.
\item[(b)] The optimality equation $J^*= T(J^*)$ holds.
\end{itemize}
\end{restatable}

\begin{restatable}[Properties of optimal policies]{thm}{thmopttwo} \label{thm-m2} 
\hfill
\begin{itemize}
\item[(a)] For any $\epsilon > 0$, there exists an $\epsilon$-optimal nonrandomized semi-Markov policy.
\item[(b)] If for each state an optimal policy exists, then there exists an optimal randomized semi-Markov policy.
\item[(c)] If $J^* \geq 0$, then there exists an $\epsilon$-optimal nonrandomized Markov policy; if in addition, for each state an optimal policy exists, then there exists an optimal nonrandomized stationary policy. Furthermore, if $J^* \geq 0$, a stationary policy $\mu$ is optimal if and only if $J^* = T_\mu(J^*)$.
\end{itemize}
\end{restatable}

The preceding theorems largely follow from the same analyses given in the paper \cite{ShrB79} and the book \cite[Chap.\ 9.2-9.3]{bs} for positive or negative costs models. We will outline the proofs in Appendix~\ref{appendix1}.

For problems with a discrete state space, proofs of Theorems~\ref{thm-m1}(b),~\ref{thm-m2}(a), as well as many other optimality results for (GC),  can be found in the survey on total reward problems by Feinberg \cite[Chap.\ 6]{Fein02}.\footnote{The survey~\cite{Fein02} focuses on GC MDP in which the state space $S$ is countable, but the control space $C$ can be an arbitrary measurable space.} (For the special positive and negative models, see also the books~\cite[Chap.\ 9]{bs} and \cite[Chap.\ 7]{puterman94}.) 

We note that the existence of a \emph{nonrandomized} $\epsilon$-optimal semi-Markov policy, stated in Theorem~\ref{thm-m2}(a), will be proved using a convergence property of value iteration in (GC) (cf.\ the discussion at the end of Section~\ref{sec3}; see also the proof given in Appendix~\ref{appsec-prf-thm2a}). This is the only property, among the ones given by the preceding theorems, whose proof depends on the properties of value iteration. Our subsequent convergence analysis for value iteration will use the optimality equation $J^* = T(J^*)$ [Theorem~\ref{thm-m1}(b)], so there is no circular reasoning involved.

The properties given in Theorem~\ref{thm-m2}(c) are known to hold for the positive model (P)~\cite[Props.\ 9.12, 9.19]{bs}, so Theorem~\ref{thm-m2}(c) shows that the (GC) model is similar to (P) in some aspects when $J^* \geq 0$. In Section~\ref{sec4.1}, we will give more results about value iteration in the case $J^* \geq 0$, which will show further connections between the two models. We also note that the last statement in Theorem~\ref{thm-m2}(c) is implied by a more general statement in Sch{$\ddot{\mathrm{a}}$}l~\cite[Theorem 5.2.2]{Schal75}, which does not require $J^* \geq 0$. 
           
\subsection{The Question about Value Iteration} \label{sec2.3}
              
In the rest of this paper we will focus on value iteration, i.e., the fixed point iteration $T^n(J)$ with the dynamic programming operator $T$, for some initial function $J \in A(S)$. The question of our interest is: for the (GC) model, under what conditions does $T^n(J)$ converge to $J^*$?  The mode of convergence we consider is pointwise convergence. For a sequence of functions $f_n$, we shall consider the functions $\limsup_{n \to \infty} f_n$, $\liminf_{n \to \infty} f_n$, where the limsup and liminf are defined pointwise. When these two functions are equal, they define the pointwise limit function $\lim_{n \to \infty} f_n$.
We write $f_n \to f$ or $\lim_{n \to \infty} f_n = f$ if $f_n$ converges pointwise to a function $f$, and we write $f_n \uparrow f$ or $f_n \downarrow f$ if the convergence is monotonically from below or from above, respectively. 

A natural function from which to start value iteration is the constant function zero, which we denote by $\0$. It can be a difficult starting point, however, as there exist many counterexamples with $T^n(\0)\not\to J^*$ (see e.g., \cite{Fein02}) and as our analysis in Section~\ref{sec4} will also show. 
An initial function greater than $J^*$ can be easier for value iteration to converge. There is also evidence that in certain problems (such as deterministic shortest path problems), value iteration can converge significantly faster from above.
We will start our analysis with such initial functions in the next section.

In this paper we focus on deriving convergence results for the (GC) model, without making further structural assumptions on the model. There are several classes of total cost MDP where additional structures can be exploited, and where the convergence of value iteration has been largely established. A major class of such problems is the one in which $T$ can be shown to be a contraction operator on a certain complete metric space and the convergence of value iteration can be established with Banach's contraction principle. This class includes certain discounted problems as well as undiscounted problems with ``transience structures'' (see e.g., \cite[Chaps.~8 and 9.6]{HL99}). Semicontinuous models cover another broad class of problems for which value iteration has been relatively well studied (see e.g., the early work \cite[Secs.\ 13-16]{Schal75} and the recent papers \cite[Sec.\ 4]{FKZ12}, \cite{FKZ14}), and for arbitrarily signed one-stage costs, the convergence of $T^n(\0)$ has been shown under some additional conditions on the costs of policies (see e.g., \cite[Condition (C)]{Schal75}). The results of this paper for the general (GC) model apply to semicontinuous models as well and can be useful as supplements to the existing results on value iteration developed by those early works just mentioned.

\section{Convergence of Value Iteration from Above} \label{sec3}
\markboth{\rm \S \ref{sec3}. Convergence of Value Iteration from Above}{\rm \S \ref{sec3}. Convergence of Value Iteration from Above}

In this section we prove a convergence theorem for value iteration starting from a function $J \geq J^*$. 
Let $J^{*+}$ (resp.\ $J^{*-}$) be the optimal cost function if $g_+$ (resp.\ $g_-$) is the one-stage cost function:
$$  J^{*+}(x): = \inf_{\pi \in \Pi} J_\pi^+(x), \qquad J^{*-}(x): = \inf_{\pi \in \Pi} J_\pi^{-}(x), \qquad x \in S.$$
Both functions are nonnegative, and $J^{*+}$ is extended real-valued, whereas $J^{*-}$ is real-valued under the model condition (\ref{eq-tmodel}). Clearly $J^{*+} \geq J^{*+} - J^{*-}  \geq J^*$.

\begin{thm}[Convergence of value iteration from above] \label{thm-vi}
For any function $J \in A(S)$ satisfying $J^* \leq J \leq c \, \big(J^{*+} + J^{*-} \big)$ for some $c \geq 1$, we have $T^n (J) \to J^*$.
\end{thm}

Theorem~\ref{thm-vi} extends a theorem of van der Wal, which states that $T^n(J) \to J^*$ if $J^* \leq J \leq J^{*+}$, in particular, $T^n(J^{*+}) \to J^*$ \cite[Theorem 3.7]{vdW81} (see also \cite[Theorem 6.8]{Fein02}). It can be seen from the definition of $T$ and $J^{*+}$ that when starting from $J^{*+}$, the convergence is monotonic: $T^n(J^{*+}) \downarrow J^*$. If the function $J^{*+} - J^{*-} \in A(S)$, then since $J^{*+} - J^{*-} \geq J^*$, it also satisfies the convergence condition and as can be easily verified, the convergence is also monotonic: 
$T^n\big(J^{*+} - J^{*-}\big) \downarrow J^*$.

Theorem~\ref{thm-vi} is also an extension of \cite[Theorem 5.1]{YuB-mvipi} for the positive model (P), which states that for (P), $T^n(J) \to J^*$ if $J^* \leq J \leq c J^{*}$ for some $c \geq 1$. This type of initial condition on value iteration resembles Whittle's bridging condition \cite{Whit79}.

To prove Theorem~\ref{thm-vi}, we will use arguments similar to those used in the proofs given in \cite[Appendix E]{YuB-mvipi} for the model (P). The line of analysis was motivated by the work of Meyn~\cite[Chap.~9]{Mey08}, who also systematically used the set of functions that are bounded by a multiple of the optimal cost function, to analyze average cost and total cost problems.

We will need the following two lemmas in the proof. Lemma~\ref{lma-B} states some basic properties of the (GC) model, and 
Lemma~\ref{lma-1} leads to the desired result. 
A subtlety here is that while the function $J^{*+}$ is lower semi-analytic and therefore universally measurable \cite[Chap.\ 7.7]{bs}, the function $J^{*-}$ is not necessarily universally measurable.
\footnote{Instead of being lower semi-analytic, the function $g_- = \max\{ - g, 0\}$ is upper semi-analytic (i.e., $- g_-$ is lower semi-analytic). For this reason, $J^{*-}$ may not be universally measurable (cf.\ \cite[Example (48), p.\ 940]{BFO74}).  If $g$ is Borel measurable, then $g_-$ would be lower semi-analytic and in turn, being the optimal cost function of the total cost problem with one-stage cost function $g_-$,  $J^{*-}$ would be lower semi-analytic and hence universally measurable.}
In the proof, we will need the outer integral of $J^{*-}$, which we introduce first.

For any nonnegative function $J$ on $S$ (possibly $J \not\in \M(S)$, i.e., non-measurable), let $\oE^{\pi}_{n,x} \big\{ J(x_n) \big\}$ denote the outer integral of $J$ with respect to the marginal distribution of $x_n$ induced by the policy $\pi$ and the initial state $x_0=x$; i.e.,
\begin{equation} \label{eq-outerint1}
  \oE^{\pi}_{n,x} \big\{ J(x_n) \big\} : = \inf \Big\{ \, \E^{\pi}_{x} \big\{ f(x_n) \big\} \, \Big| \,  J \leq f, \  f  \in \M(S) \Big\}.
\end{equation}  
By a property of the outer integral \cite[Lemma A.2]{bs}, if $J_1, J_2 \geq 0$ and $J_1 \in \M(S)$, then
\begin{equation} \label{eq-outerint2}
  \oE^{\pi}_{n,x} \big\{ J_1(x_n) + J_2(x_n) \big\}  =  \E^{\pi}_{x} \big\{ J_1(x_n) \big\}  + \oE^{\pi}_{n,x} \big\{ J_2(x_n) \big\}.
\end{equation}

\begin{restatable}[Two basic properties]{lem}{lmaB} 
\label{lma-B} \hfill
\begin{itemize}
\item[(a)] For any given state $x$ and policy $\pi$, there exists a Markov policy $\pi'$ with 
$$ J_{\pi'}^+(x)=J_{\pi}^+(x), \qquad J_{\pi'}^-(x)=J_{\pi}^-(x), \qquad J_{\pi'}(x) = J_\pi(x).$$
\item[(b)] For any state $x$, policy $\pi$, and $n \geq 0$, 
$$\E^\pi_x \left\{ \sum_{k=n}^\infty g(x_k, u_k) \right\} \geq \E^\pi_x \left\{ J^*(x_n) \right\}, \qquad  \E^\pi_x \left\{ \sum_{k=n}^\infty g_-(x_k, u_k) \right\} \geq \oE^\pi_{n,x} \left\{ J^{*-}(x_n) \right\}.$$
\end{itemize}
\end{restatable}

We give the proof of the preceding lemma in Appendix~\ref{appendix2}.

\begin{lem} \label{lma-1}
For a given $x \in S$, if $\pi$ is a policy with $J_\pi(x) < + \infty$, then 
$$\lim_{n \to \infty} \, \E^{\pi}_x \big\{ J^{*+}(x_n) \big\} = 0, \qquad \lim_{n \to \infty} \, \oE^{\pi}_{n,x} \big\{ J^{*-}(x_n) \big\} = 0.$$
\end{lem}
\begin{proof}
Since $J_\pi(x) = J^{+}_\pi(x) - J^{-}_\pi(x) < \infty$ by assumption and $J_\pi^{-}(x) < \infty$ by the model condition (\ref{eq-tmodel}), we have $J_\pi^+(x) < \infty$. 
For any $n \geq 0$,
\begin{align*}
  J_\pi^+(x) & =  \E^\pi_x \left\{ \sum_{k=0}^{n-1} g_+(x_k, u_k) \right\} + \E^\pi_x \left\{ \sum_{k=n}^{\infty} g_+(x_k, u_k) \right\} \\
     & \geq  \E^\pi_x \left\{ \sum_{k=0}^{n-1} g_+(x_k, u_k) \right\} +  \E^\pi_x \big\{ J^{*+}(x_n) \big\}.
\end{align*} 
where the inequality follows from applying Lemma~\ref{lma-B}(b) to the total cost problem that has $g_+$ as the one-stage cost function.
As $n \to \infty$, $\E^\pi_x \left\{ \sum_{k=0}^{n-1} g_+(x_k, u_k) \right\} \uparrow J_\pi^+(x) < \infty$. So the preceding inequality implies that
$\limsup_{n \to \infty} \, \E^\pi_x \big\{ J^{*+}(x_n) \big\} = 0$ and hence $\lim_{n \to \infty} \, \E^\pi_x \big\{ J^{*+}(x_n) \big\} = 0$.
The second statement for $J^{*-}$ is proved similarly. We start with the inequality
\begin{align*} 
 J_\pi^-(x) & =  \E^\pi_x \left\{ \sum_{k=0}^{n-1} g_-(x_k, u_k) \right\} + \E^\pi_x \left\{ \sum_{k=n}^{\infty} g_-(x_k, u_k) \right\}  \\
   & \geq   \E^\pi_x \left\{ \sum_{k=0}^{n-1} g_-(x_k, u_k) \right\} + \oE^\pi_{n,x} \big\{ J^{*-}(x_n) \big\},
\end{align*} 
which follows from Lemma~\ref{lma-B}(b). 
Since as $n \to \infty$, 
$\E^\pi_x \left\{ \sum_{k=0}^{n-1} g_-(x_k, u_k) \right\}  \uparrow J_\pi^-(x) < \infty$, we obtain from the preceding inequality $\lim_{n \to \infty} \, \oE^\pi_{n,x} \big\{ J^{*-}(x_n) \big\} = 0$.
\end{proof}

\begin{proof}[Proof of Theorem~\ref{thm-vi}]
Let $J$ be a function satisfying the condition of the theorem, i.e., $J \in A(S)$ and for some $c \geq 1$, $J^* \leq J \leq c (J^{*+} + J^{*-})$. Since $J \geq J^*$, $T^n(J) \geq T^n(J^*)=J^*$ for all $n$ by the monotonicity of $T$. So to prove $T^n(J) \to J^*$, it is sufficient to consider an arbitrary state $x$ with $J^*(x) < +\infty$ and show that $\limsup_{n \to \infty} T^n(J)(x) \leq J^*(x)$.

Let $\epsilon > 0$ and let $\pi=(\mu_0, \mu_1, \ldots)$ be an $\epsilon$-optimal Markov policy for the state $x$; such a policy exists by Lemma~\ref{lma-B}(a). For any $n \geq 1$, by the definition of $T$ and the monotonicity of $T$, we have
\begin{align}
 T^n (J)(x) & \leq \big(T_{\mu_0} \circ T_{\mu_1} \circ \cdots T_{\mu_{n-1}} \big) (J)(x) \notag \\
       & = \E^{\pi}_x \left\{ \sum_{k=0}^{n-1} g(x_k, u_k) \right\} + \E^\pi_x \big\{ J(x_n) \big\} \notag \\
       & \leq \E^{\pi}_x \left\{ \sum_{k=0}^{n-1} g(x_k, u_k) \right\} + c \, \oE^\pi_{n,x} \big\{ J^{*+}(x_n) + J^{*-}(x_n) \big\}, \label{eq-prf-altprfP} \\
       & = \E^{\pi}_x \left\{ \sum_{k=0}^{n-1} g(x_k, u_k) \right\} + c \, \E^\pi_x \big\{ J^{*+}(x_n) \big\}  + c \, \oE^\pi_{n,x} \big\{ J^{*-}(x_n) \big\}, \label{eq-prf-altprfP2}
\end{align}
where we used the assumption $J \leq c \, (J^{*+} + J^{*-})$ and the definition (\ref{eq-outerint1}) of outer integral in the inequality~(\ref{eq-prf-altprfP}), and we used the fact $J^{*+} \in \M(S)$ and the equality~(\ref{eq-outerint2}) in Eq.~(\ref{eq-prf-altprfP2}).
By the $\epsilon$-optimality of $\pi$ for $x$, 
$$ \lim_{n \to \infty} \, \E^{\pi}_x \left\{ \sum_{k=0}^{n-1} g(x_k, u_k) \right\} = J_{\pi}(x) \leq J^*(x) + \epsilon < \infty.$$
Applying Lemma~\ref{lma-1} and using the fact $J_\pi(x) < \infty$, we also have
$\lim_{n \to \infty} \E^\pi_x \big\{ J^{*+}(x_n) \big\} = 0$ and $\lim_{n \to \infty} \oE^\pi_{n,x} \big\{ J^{*-}(x_n)  \big\} = 0$.
Combining these relations with the inequality~(\ref{eq-prf-altprfP2}), we obtain
$$  \limsup_{n \to \infty} \,T^n (J)(x) \leq J^*(x) + \epsilon.$$
Since $\epsilon$ is arbitrary, this gives $\limsup_{n \to \infty} T^n (J)(x) \leq J^*(x)$. Hence $T^n (J)(x) \to J^*(x)$ for all $x$, as discussed earlier. 
\end{proof}

\begin{rem} 
We have stated in Theorem~\ref{thm-m2}(a) that for the (GC) model, an $\epsilon$-optimal semi-Markov policy can be chosen to be nonrandomized. The construction of such a policy will use the fact that value iteration converges from above (for example, from $J^{*+}$). 
For details, see Appendix~\ref{appsec-prf-thm2a}. \qed
\end{rem}
\mysmallskip

Let $J^s$ be the minimal cost achievable by nonrandomized stationary policies: 
$$ \qquad \qquad J^s(x) := \inf_{\pi \in \Pi_s} J_\pi(x), \quad x \in S, \quad \text{where} \ \Pi_{s} = \big\{ \pi \in \Pi \mid \pi \ \text{nonrandomized and statioanry} \big\}. $$
We say $\Pi_s$ is \emph{locally adequate} if $J^s(x) = J^*(x)$ for all $x \in S$~\cite{ScS87}.
When the state space $S$ is countable, it is known that $J^s = T(J^s)$ (Feinberg and Sonin~\cite[Theorem 2.2]{FeS83}; see also \cite[Theorem 6.11]{Fein02}), so by Theorem~\ref{thm-vi} we obtain the following corollary about $J^s$.

\begin{cor} \label{cor-Js}
Suppose $S$ is countable. If $\Pi_s$ is not locally adequate, i.e., $J^s \not = J^*$, then for any arbitrarily large $c>0$, there exists some state $x$ with $J^s(x) > c \big(J^{*+}(x) + J^{*-}(x)\big)$.
\end{cor}

For an example problem that illustrates Cor.~\ref{cor-Js}, see e.g.,~\cite[Example 6.5]{Fein02} (in which there exists a state $x$ with a non-compact control set such that $J^s(x) =1$ but $J^{*+}(x) + J^{*-}(x) =0$). Let us also mention in passing that regarding the local adequacy of nonrandomized stationary policies, it is known that $J^s=J^*$ if $S$ is countable and all the control sets $U(x)$ are finite (van der Wal~\cite[Theorem 2.22]{vdW81}), or if $S$ is a Borel space and either $g \leq 0$ (Blackwell~\cite{Blk-positive}) or certain continuity and compactness conditions are satisfied (Sch{$\ddot{\mathrm{a}}$}l~\cite{Sch83}). For countable $S$, it is also known that $J^s = J^*$ under much weaker conditions (see Feinberg~\cite[Chaps. 6.7-6.9]{Fein02}). 

\subsubsection*{The Special Case (UD)}

For discounted problems in the (UD) class (cf.\ Def.~\ref{def-ud}), the dynamic programming operator $T$ is given by Eq.~(\ref{eq-T}), except that the integral in (\ref{eq-T}) is multiplied by the discount factor $\beta$:
$$   T(J)(x)  = \inf_{u \in U(x)} \left\{ g(x,u) + \beta \int_S J(x') \, q(dx' \! \mid x, u) \right\},  \qquad x \in S, \ \forall \, J \in A(S).$$
 As noted by Whittle in the case of the positive costs model~\cite{Whit79}, for $\beta$-discounted problems, 
if $T^n(J) \to J^*$, then for any constant $b$, since $T^n(J+b) = T^n(J) + \beta^n b$, we also have $T^n(J+b) \to J^*$. 
Thus, the range of initial functions in Theorem~\ref{thm-vi} can be enlarged, leading to stronger conclusions:

\mysmallskip
\begin{cor} \label{cor-ud1}
{\rm (UD)} \ For any function $J \in A(S)$ satisfying $J^* - b \leq J \leq c \, (J^{*+} + J^{*-})  + b$ for some scalars $c \geq 1$, $b \geq 0$, we have $T^n (J) \to J^*$. In particular, if $J^*$ is bounded above, then $T^n(\0) \to J^*$.
\end{cor}

\section{Convergence of Value Iteration starting from Zero} 
\label{sec4}
\markboth{\rm \S \ref{sec4}. Convergence of Value Iteration from Zero}{\rm \S \ref{sec4}. Convergence of Value Iteration from Zero}

In this section we tackle the question of the convergence of $T^n(\0)$ to $J^*$.
The definition of $J^*$ implies\footnote{For any policy $\pi$, $\limsup_{n \to \infty} \, T^n(\0) \leq J_\pi$, from which the inequality~(\ref{eq-ineq0}) follows.}
\begin{equation} \label{eq-ineq0}
 J_\infty:= \limsup_{n \to \infty} \, T^n(\0) \leq J^*,
\end{equation}
so the convergence of $T^n(\0)$ in question amounts to $\liminf_{n \to \infty} T^n(\0)  \geq J^*$.

The behavior of $T^n(\0)$ is fairly complex because the one-stage cost function can take both positive and negative values.
It is known that for the (GC) model:
\begin{itemize}
\item[(i)] $\{T^n(\0)\}$ need not converge to $J^*$. It may not even have a pointwise limit. (See e.g., \cite[Example 6.11]{Fein02} or the subsequent Example~\ref{ex-nolimit} in Section~\ref{sec5}.)
\item[(ii)] Even if $J_\infty = \limsup_{n \to \infty} T^n(0)$ is a fixed point of $T$ (and the limit of $T^n(\0)$), it is still possible that $J_\infty \not= J^*$. This can happen even for finite state and control problems. (See e.g., \cite[Example 6.10]{Fein02} or the subsequent Example~\ref{ex-rel2P} in Section~\ref{sec4.1}.) 
\end{itemize}

In what follows we first discuss two special cases, $J^* \leq 0$ and $J^* \geq 0$, in which the behavior of $T^n(\0)$ is shown to be similar to its behavior in the negative model and in the positive model, respectively. For the case $J^* \geq 0$, we will give several results about fixed point properties of $T$ and convergence properties of transfinite value iteration. These results suggest that when $J^* \geq 0$, the (GC) model is similar to the positive costs model in many aspects. They lead us further to prove the convergence of non-transfinite value iteration for finite state or finite control problems, where value iteration employs the operator $\max\{ \0, T(\cdot)\}$, which is derived from our analysis of transfinite value iteration. 

We then focus on the general case where $J^*$ may take both positive and negative values. We give sufficient conditions for the convergence of $T^n(\0)$ and more generally, sufficient conditions for the convergence of $T^n(\0)(x)$ to $J^*(x)$ for some states $x$. These conditions can be viewed as a generalization of the bridging condition of Whittle~\cite{Whit79} from the positive cost model to the more general class of total cost problems considered here. We shall discuss this connection after presenting the convergence results. A special case of our results which relates to two existing sufficient conditions for value iteration will also be discussed then.

\subsection{Special Cases where $J^*$ is Nonpositive or Nonnegative} \label{sec4.1}

First, we observe that if $J^* \leq 0$, the problem is similar to the negative costs model (N), in the sense that $T^n(\0) \to J^*$.
This follows from Theorem~\ref{thm-vi}, but a direct proof is simpler: when $J^* \leq 0$, by the monotonicity of $T$, $T^n(\0) \geq T^n(J^*)=J^*$, whereas $\limsup_{n \to \infty} T^n(\0) \leq J^*$.

If $J^* \geq 0$, the (GC) model exhibits a number of properties analogous to those in the positive costs model (P), which will be the focus of our discussion below. As Theorem~\ref{thm-m2}(c) already showed, in the (GC) model, when $J^* \geq 0$, we have the same structural properties of optimal policies as in the positive costs model. What we are going to discuss now are the fixed point properties of $T$ and the convergence properties of value iteration starting from zero, when $J^* \geq 0$.

Let us first recall that in the positive costs model, if $J$ is a nonnegative lower semi-analytic function on $S$ and $J \geq T(J)$, then $J \geq J^*$, so $J^*$ is the unique fixed point of $T$ within the set $\{ J \in A(S) \mid 0 \leq J \leq J^*\}$ \cite[Prop.\ 9.10]{bs}.
The (GC) model has a similar property:

\begin{prop} \label{prp-relate2P}
If a function $J \in A(S)$ satisfies $J \geq 0$ and $J \geq T(J)$, then $J \geq J^*$. 
\end{prop}
\begin{proof}
Under the (GC) condition~(\ref{eq-tmodel}), $T(\0) > - \infty$, so by the monotonicity of $T$ and the assumption $J \geq 0$,  $T(J)(x) \geq T(\0)(x) > - \infty$ for all $x \in S$. It then follows from the selection theorem~\cite[Prop.\ 7.50]{bs} that given any $\epsilon > 0$, there exist stationary policies $\mu_k$, $k \geq 0$, such that
\begin{equation} \label{eq-prf-rel2Pa}
   T_{\mu_k}(J) \leq T (J) + 2^{-k-1} \epsilon \leq J + 2^{-k-1} \epsilon,
\end{equation}   
where the second inequality follows from the assumption $T(J) \leq J$.
Consider the Markov policy $\pi = (\mu_0, \mu_1, \ldots)$. By Eq.~(\ref{eq-prf-rel2Pa}) and the monotonicity of the mappings $T_{\mu_0}, T_{\mu_1}, \ldots$, we have
$$ T_{\mu_0}(J)  \leq  J + 2^{-1}  \epsilon, \qquad \big(T_{\mu_0} \circ T_{\mu_1} \big)(J) \leq J + (2^{-1} + 2^{-2}) \cdot \epsilon, \qquad \ldots$$
and consequently,
$$  \big(T_{\mu_0} \circ T_{\mu_1} \circ \cdots  T_{\mu_k} \big) (\0) \leq \big(T_{\mu_0} \circ T_{\mu_1} \circ \cdots  T_{\mu_k} \big) (J)  \leq J  + \epsilon, \qquad \forall \, k \geq 1,$$
where we used the assumption $J \geq 0$ in the first inequality. As $k \to \infty$, the first expression in the preceding inequality converges to $J_{\pi}$ in the (GC) model; therefore,
$J^* \leq J_\pi \leq J + \epsilon$ and equivalently, $J^* \leq J$ since $\epsilon$ is arbitrary. 
\end{proof}

The next two corollaries are direct consequences of the preceding proposition. 
The second corollary uses also the fact that $J_\infty = \limsup_{n \to \infty} T^n(\0) \leq J^*$.

\begin{cor} \label{cor-gc1}
If $J^* \geq 0$, then within the set $\{ J \in A(S)  \mid 0 \leq J \leq J^*  \}$, $J^*$ is the unique fixed point of $T$.
\end{cor}

\begin{cor} \label{cor-gc2}
If $J_\infty$ is a fixed point of $T$ and $J_\infty \geq 0$, then $J_\infty = J^*$.
Hence, if $J_\infty$ is a fixed point of $T$ and yet $J_\infty \not=J^*$, then $J_\infty(x) < 0$ for some state $x$.
\end{cor}

For the discounted problems in the (UD) class (cf.\ Def.~\ref{def-ud}), by the same reasoning given at the end of Section~\ref{sec3}, we can strengthen Cors.~\ref{cor-gc1},~\ref{cor-gc2} to obtain stronger conclusions. They are given by the next two corollaries, with the second one being a consequence of the first.

\begin{cor} \label{cor-ud3}
{\rm (UD)} \ If $J^* \geq 0$, then $J^*$ is the unique fixed point of $T$ within the set 
$$\{ J \in A(S)  \mid -b \leq J \leq J^* +b, \ \text{for some scalar} \ b \geq 0  \}.$$ 
\end{cor}

\begin{proof}
Let $\bar J$ be a fixed point of $T$ in the set of functions stated by the corollary. Let $\beta$ be the discount factor. We show that $\bar J = J^*$. Let $\hat J = \bar J +b \geq 0$. Then $T(\hat J) = T(\bar J) + \beta b = \bar J + \beta b \leq \hat J$, so by Prop.~\ref{prp-relate2P}, $\hat J \geq J^*$. Consequently, $T^n(\hat J) \geq J^*$ for all $n$. We also have $T^n(\hat J) \to \bar J$ since $T^n(\hat J) = T^n(\bar J) + \beta^n b = \bar J + \beta^n b$. Therefore, $\bar J \geq J^*$. On the other hand, since $\bar J \leq J^* +b$, we have $\bar J = T^n(\bar J) \leq T^n(J^* + b) \to J^*$. Thus, $\bar J = J^*$.
\end{proof}

\begin{cor} \label{cor-ud4}
{\rm (UD)} \ If $J^* \geq 0$ and if in addition, the function $J_\infty$ 
is a fixed point of $T$ and yet $J_\infty \not=J^*$, then $\inf_{x \in S} J_\infty(x) = - \infty$. 
\end{cor}

\subsubsection*{Transfinite Value Iteration}

Transfinite value iteration was introduced and analyzed by Maitra and Sudderth~\cite{MS92} for the positive costs model in the universal measurability framework. In an abstract dynamic programming context, it was also briefly discussed in~\cite[p.\ 463]{KreP77}. Transfinite value iteration is defined by transfinite recursion---recursion defined on a general well-ordered set; by comparison, value iteration is ordinarily defined by the classical recursion on the set of nonnegative integers. The need for studying transfinite value iteration comes from the convergence difficulty of ordinary value iteration.

The well-ordered set $(X, <)$ involved in transfinite value iteration is a set of ordinals, denoted $\{0, 1, \ldots, \omega_1\}$, with $\omega_1$ being the first uncountable ordinal.
\footnote{We give in Appendix~\ref{appendix3-1} an informal introduction of ordinals. As a full definition of ordinals is beyond the scope of this paper, we refer the readers to the books \cite[Chaps.\ 1.3, A.3]{Dud02} and \cite[p.\ 27-28]{Kur-topology} for the theory of ordinals, and their application in transfinite recursion and induction. See \cite[Chaps.\ 1.3]{Dud02} also for the general principles of recursion and induction on well-ordered sets, which extend the principles of the classical recursion and induction on the set of nonnegative integers.}  
The transfinite value iteration for the positive costs model \cite{MS92} can be described as follows: for ordinals $\xi \leq \omega_1$, define recursively
\begin{equation} \label{eq-def-transvip}
   T^0(J) = J; \qquad T^\xi(J) = T \Big( \sup _{\eta < \xi} \, T^\eta (J) \Big), \quad 0 < \xi < \omega_1; \qquad T^{\omega_1}(J) = \sup _{\xi < \omega_1} \, T^\xi (J).
\end{equation}   
It follows from the recursion principle on a well-ordered set \cite[Theorem 1.3.2]{Dud02} and the properties of the universal measurability framework
that the preceding recursion is well-defined, with $T^\xi(J) \in A(S)$ for all $\xi < \omega_1$.

We define transfinite value iteration similarly for the (GC) model. It will be more convenient for us to work with a mapping $\T: A(S) \to A(S)$, which is a modification of the dynamic programming operator $T$, and which has the desired increasing property $\T(J) \geq J$:
\begin{equation} \label{eq-def-tT}
\T(J) : = \max \big\{ J  , \, T(J) \big\}.
\end{equation}
We use $\T$ to define transfinite value iteration as follows. 
For ordinals $\xi < \omega_1$, define recursively
\begin{equation} 
    \T^0(J) = J, \qquad \T^1(J) = \T(J),  \notag 
\end{equation}    
and
\begin{equation} \label{eq-transvi0}
     \T^\xi(J) =  \T \Big( \sup _{\eta < \xi} \, \T^\eta (J) \Big), \qquad 1 < \xi < \omega_1.
\end{equation}    
We then let
\begin{equation}
     \T^{\omega_1} (J) = \sup_{\xi < \omega_1} \T^\xi (J). \label{eq-transvi1}
\end{equation}     
Similar to the recursion (\ref{eq-def-transvip}), the preceding recursion is well-defined, with $\T^\xi(J) \in A(S)$ for all $\xi < \omega_1$, by the recursion principle on a well-ordered set \cite[Theorem 1.3.2]{Dud02} and by the properties of the universal measurability framework.

Note that since $J^* = T(J^*)$, we have $\T(J^*) = J^*$ and hence $\T^{\omega_1} (J^*) = J^*$. 
Note also that by the increasing property $\T(J) \geq J$, the transfinite value iterates are nondecreasing:
$$ \T^\xi(J) \geq \T^\eta(J) \quad \text{if} \  \ \xi > \eta.$$ 
Using the monotonicity of $T$, we can also write Eq.~(\ref{eq-transvi0}) equivalently as
\begin{equation} \label{eq-tT-alt}
  \T^\xi(J) = \max\big\{J \, , \, T(\tilde J_\xi) \big\}, \qquad \text{where} \ \ \  \tilde J_\xi =  \sup_{\eta < \xi} \, \T^\eta (J).
\end{equation}  
For an integer $n \geq 0$, in general $\T^n(J) \not= T^n(J)$ (the ordinary value iteration). However, if $J$ satisfies $T(J) \geq J$, then by the monotonicity of $T$, the two iterations become identical:
$$\T(J) = \max \big\{ J, \, T(J) \big\} = T(J), \qquad \T^2(J) = \max \big\{ \T(J), \, T(\T(J)) \big\} = T^2(J), \ \ \ldots,$$ 
and so do $\T^\xi(J)$ and $T^\xi(J)$ for all ordinals $\xi \leq \omega_1$.

For the positive costs model (P), Maitra and Sudderth~\cite{MS92} showed that transfinite value iteration starting from zero converges to $J^*$, with $T^{\omega_1}(\0)=J^*$. Using their proof arguments together with Prop.~\ref{prp-relate2P}, we can show that 
the (GC) model has a similar property:

\begin{restatable}[Convergence of transfinite value iteration]{thm}{thmreltoP} \label{thm-relate2P} 
If $J^* \geq 0$, then $\T^{\omega_1}(\0) = J^*$.
\end{restatable}

\begin{proof}
Denote $J_{\omega_1} = \T^{\omega_1}(\0)$. We use a critical lemma, Lemma~\ref{lma-rel2P}, given in Appendix~\ref{appendix3-2}, which states that the function $J_{\omega_1}$ is lower semi-analytic and satisfies 
$$J_{\omega_1} = \T(J_{\omega_1}) \geq T(J_{\omega_1}).$$
Combining this lemma with Prop.~\ref{prp-relate2P} and the assumption $J^* \geq 0$, we can obtain \cref{thm-relate2P} immediately as follows. 
Since $J^* \geq 0$, using the definition of transfinite value iteration and the monotonicity of $T$, and using also the fact $J^* = T(J^*)$, we have
$$J_{\omega_1}  = \T^{\omega_1}(\0) \leq \T^{\omega_1}(J^*) = J^*.$$
On the other hand, by the nondecreasing property of the transfinite value iterates, we have $J_{\omega_1} = \T^{\omega_1}(\0) \geq \T^0(\0) = \0$, 
and by \cref{lma-rel2P}, we also have $J_{\omega_1} \geq T(J_{\omega_1})$ and $J_{\omega_1}$ is lower semi-analytic. So applying Prop.~\ref{prp-relate2P} with $J = J_{\omega_1}$, we have $J_{\omega_1} \geq J^*$, and therefore, we must have $J_{\omega_1} = J^*$.
\end{proof}

We mention that if the state and control spaces $S, C$ are countable, then Theorem~\ref{thm-relate2P} holds with $\T^{\xi}(\0) = J^*$ for some countable ordinal $\xi < \omega_1$. This follows from the proof of \cite[Theorem 5.1]{MS92}.
\footnote{In fact, as can be verified from its proof, the conclusion of \cite[Theorem 5.1]{MS92} also holds in our case when $S$ or $C$ is uncountably infinite. Specifically, we have that $\T^{\xi}(\0) = J^*$ for some $\xi < \omega_1$ if $J^* \geq 0$ and if there exists a probability measure $\rho \in \P(S)$ such that for every $(x, u) \in S \times C$, $q(dx' \!\mid x, u)$ is absolutely continuous with respect to $\rho$.}

If $S$ is finite, then actually the sequence $\{\T^n(\0)\}$ converges to $J^*$ if $J^*$ is nonnegative and real-valued. We prove this result below. It is similar to \cite[Cor.\ 5.1]{YuB-mvipi}, and the counterexample given immediately after \cite[Cor.\ 5.1]{YuB-mvipi} in that paper for the (P) model is also applicable here and shows that the desired convergence need not hold if $J^*$ can take the value $+\infty$. (For uncountable $S$, the proof given below is applicable if $\{\T^n(\0)\}$ can be shown to converge uniformly.)

\begin{prop}[Convergence of value iteration for finite $S$] \label{prp-vi-pos1}
Suppose the state space $S$ is finite. If $J^* \geq 0$ and $J^*$ is real-valued, then $\T^n(\0) \to J^*$.
\end{prop}
\begin{proof}
By the definition of $\T$, the sequence $\{\T^n(\0)\}$ is nondecreasing and converges monotonically to some limit $\bar J \geq 0$. Since $J^* \geq 0$, using the monotonicity of $\T$ and the fact $J^* = \T(J^*)$, we have $\bar J \leq J^*$, so $\bar J < + \infty$, in view of the assumption $J^* < + \infty$. Then, since $S$ is finite and $\bar J$ is real-valued, the convergence $\T^n(\0) \to \bar J$ must be uniform; consequently, given any $\epsilon > 0$, $\bar J \leq \T^n(\0) + \epsilon$ for all $n$ sufficiently large. Using this and the definition of $\T$, we obtain for any given $\epsilon > 0$ that for all $n$ sufficiently large,
$$ \T\big(\bar J \big) \leq \T \big( \T^n(\0) + \epsilon \big) \leq \T^{n+1}(\0) + \epsilon.$$
This implies $\T(\bar J) \leq \bar J$; hence by Prop.~\ref{prp-relate2P}, $\bar J \geq J^*$. But $\bar J \leq J^*$ as discussed earlier, so $\bar J = J^*$. 
\end{proof}

We have used $\T$ instead of $T$ to define transfinite value iteration for the (GC) model. It is not hard to verify that the transfinite value iteration defined by Eq.~(\ref{eq-def-transvip}) is essentially equivalent to the iteration (\ref{eq-transvi0})-(\ref{eq-transvi1}), and therefore, by Theorem~\ref{thm-relate2P}, $T^{\omega_1}(\0) = J^*$ if $J^* \geq 0$. However, if instead we define transfinite value iteration more similarly to ordinal value iteration as 
$$  \hat{J}_1 = T(\0), \qquad \hat{J}_\xi = T\Big(\sup_{1 \leq \eta < \xi} \hat{J}_\eta \Big), \ \ 1 < \xi < \omega_1, \qquad \hat{J}_{\omega_1} = \sup_{ 1 \leq \xi < \omega_1} \hat{J}_\xi,$$
would we still have $\hat{J}_{\omega_1} = J^*$ if $J^* \geq 0$? 
The next example shows that it is not true in general, even if $S$, $C$ are finite.
\footnote{Note, however, that finite state and control MDP under the total cost criterion can be satisfactorily solved using a policy iteration algorithm based on the sensitive optimality concept \cite[Chap.\ 10.4]{puterman94}.}
This example is a slight modification of \cite[Example 3.2]{vdW81} (cf.\ \cite[Example 6.10]{Fein02}).

\begin{example} \label{ex-rel2P} \rm
Let $S=\{0,1,2\}$. State $0$ is cost-free and absorbing. From state $1$, the system moves to state $0$ deterministically, with cost $1$. State $2$ has two controls: under one control the system remains at state $2$ with no cost, and under the other the system moves to state $1$ with cost $-1$. The optimal costs at these states are $J^*(0)=0, J^*(1) = 1, J^*(2) = 0$, so $J^* \geq 0$. However, for state~$2$, $\hat{J}_{\omega_1}(2) = -1$, since for all $1 \leq \xi < \omega_1$, $\hat{J}_{\xi}(2)=-1$. Indeed, the function $\hat{J}_1=T(\0)$, which has values $0,1,-1$ for states $0,1,2$, respectively, is a fixed point of $T$. But it is easy to see that $\T(\0) = J^*$. 
\qed
\end{example}

The next corollary follows from Theorem~\ref{thm-relate2P}. 
(Note that $J_\infty \in A(S)$ by \cite[Lemma 7.30(2)]{bs}.)

\begin{cor}
If $J_\infty \geq 0$, 
then $\T^{\omega_1}(J_\infty)  = J^*$.
\end{cor}
\begin{proof} 
Since $J^* \geq J_\infty \geq 0$, using Theorem~\ref{thm-relate2P} and the monotonicity of transfinite value iteration, we have that $J^* = \T^{\omega_1}(\0) \leq \T^{\omega_1}(J_\infty) \leq J^*$.
This shows $\T^{\omega_1}(J_\infty) = J^*$. 
\end{proof}

Finally, for the case $J^* \geq 0$, we prove the convergence of the sequence $\{\T^n(\0)\}$ when the control sets are finite. This result can be compared with the convergence of $T^n(\0)$ in positive costs problems with finite controls (cf.\ \cite{b77}, \cite[Prop.\ 9.18(P)]{bs}). 

\begin{prop}[Convergence of value iteration in the case of finite controls] \label{prp-vi-pos2}
Suppose that the control set $U(x)$ is finite for each $x \in S$. Then $\T^n(\0) \to J^*$ if $J^* \geq 0$.
\end{prop}
\begin{proof}
Denote $J_n = \T^n(\0)$ with $J_0=\0$. Similarly to the proof of Prop.~\ref{prp-vi-pos1}, we have that $J_n \geq 0$ and $J_n \uparrow \bar J$ for some nonnegative function $\bar J \in A(S)$ with $\bar J \leq J^*$.  Consider an arbitrary state $x$. Since $U(x)$ is finite, there exist a control $\bar u \in U(x)$ and a subsequence $\{n_k\}$ of positive integers such that for all $k \geq 1$, $\bar u$ is the control that attains the minimum in $T\big(J_{n_k}\big)(x)$, i.e.,
\begin{equation} \label{eq-prf-pos2a}
  J_{n_k+1}(x) = \T\big( J_{n_k} \big)(x) =  \max\big\{ J_{n_k}(x)  , \, T\big(J_{n_k}\big)(x) \big\} = \max \big\{ J_{n_k}(x)  , \, H \big(x, \bar u, J_{n_k}) \big\},
\end{equation} 
where
$$ H \big(x, \bar u, J_{n_k}) : = g(x, \bar u) + \int_S J_{n_k}(x') \, q(dx'\!\mid x, \bar u) = \inf_{u \in U(x)}  \left\{ g(x, u) +  \int_S J_{n_k}(x') \, q(dx'\!\mid x, u) \right\}.$$
Since $J_{n_k} \geq 0$ and $J_{n_k} \uparrow \bar J$, we have $\lim_{k \to \infty} \int_S J_{n_k}(x') \, q(dx'\!\mid x, \bar u)  = \int_S \bar J(x') \, q(dx'\!\mid x, \bar u)$ by the monotone convergence theorem. Hence, from Eq.~(\ref{eq-prf-pos2a}), by letting $k$ go to $+\infty$ and then applying the inequality $g(x, \bar u) + \int_S \bar J(x') \, q(dx'\!\mid x, \bar u) \geq T(\bar J)(x)$,
we have
$$ \bar J(x)  \geq \max \big\{ \bar J(x) , \, T\big(\bar J\big) (x) \big\} \quad \Longrightarrow \quad \bar J(x) \geq T\big(\bar J\big)(x).$$
Since $x$ is arbitrary, this proves $\bar J \geq T(\bar J)$, which together with the fact $\bar J \geq 0$ implies $\bar J \geq J^*$ by Prop.~\ref{prp-relate2P}. But $\bar J \leq J^*$ as we showed earlier, so $\bar J = J^*$.
\end{proof}

Proposition~\ref{prp-vi-pos2} gives us a practical way of computing $J^*$ for finite control problems when we can infer $J^* \geq 0$ from the problem data relatively easily without intensive computation. We also note that the iteration $J_{n+1} = \T(J_n)$ with $J_0=\0$ can be equivalently written as the slightly simpler iteration $J_{n+1} = \max\{ \0,\, T(J_n)\}$, as mentioned earlier in Eq.~(\ref{eq-tT-alt}) and can be verified by induction. 

\subsection{The General Case} \label{sec4.2}

We now consider the convergence of $T^n(\0)$ in the general case where $J^*$ can take both positive and negative values.
We will formulate sufficient convergence conditions, using some ideas from Whittle's bridging condition for the positive costs model (Whittle~\cite{Whit79}, Hartley~\cite{Har80}). Our conditions can characterize convergence of $T^n(\0)$ on parts of the state space, so they can be useful even if on the full state space, $T^n(\0) \not\to J^*$. These conditions can be viewed as generalizations of the bridging condition from the positive model to the more general (GC) model, and we will discuss this connection after giving our convergence results (see Remark~\ref{rmk-bridging} in Section~\ref{sec4.2-discussion}).

We start by introducing a few mappings, function spaces and inequalities, which we will need in analyzing the convergence of $\{T^n(\0)\}$ and bounding these iterates from below.

\subsubsection{Preliminaries}

 Define two dynamic programming operators $T_+ : A(S) \to A(S)$, $T_-: A(S) \to A(S)$, corresponding to the positive and negative parts of $g$, respectively, as follows: for any $J \in A(S)$,
\begin{align*}
  T_+(J)(x) & : = \inf_{u \in U(x)} \left\{ g_+(x,u) + \int_S J(x') \, q(d x' \!\mid x, u) \right\}, \qquad x \in S, \\
  T_-(J)(x) & : = \inf_{u \in U(x)} \left\{ - g_-(x,u) + \int_S J(x') \, q(d x' \!\mid x, u) \right\}, \qquad x \in S.
\end{align*}  
Also define $T_0: A(S) \to A(S)$ by
\begin{equation}
T_0(J)(x) : = \inf_{u \in U(x)}  \int_S J(x') \, q(d x' \!\mid x, u), \qquad x \in S, \ \forall \, J \in A(S). \notag
\end{equation}
We may treat $\infty - \infty = \infty$ in the above definitions. However, for the inequalities we will need, we will focus on a smaller set of functions, $A_0(S) \subset A(S)$, so that $\infty - \infty$ does not occur when applying any of the above mappings to a function in the set. We define this set as follows.

Let $A_0(S)$ be a subset of lower semi-analytic functions given by
\begin{align}
   A_0(S) : = \left\{ J \in A(S) \ \Big| \ J > - \infty, \ \inf_{n \geq 1} \, T_0^n\big( \min\{J, 0\} \big) > - \infty \,  \right\},
\end{align}
where the infimum over $n$ is taken pointwise.
The set $A_0(S)$ has the desired property that if $J_1, J_2 \in A_0(S)$, then in the calculation of $J_1 + J_2$ or
the integral $\int_S \big( J_1(x') + J_2(x') \big) \, q(d x' \!\mid x, u)$, we never encounter $\infty - \infty$ or $-\infty + \infty$, and moreover, with $J = J_1 + J_2$, the linearity property of expectation can be taken for granted:
\footnote{This is not the case if $J_1$ or $J_2$ is not in the set $A_0(S)$. As an example, suppose $J_1 = f$ and $J_2= -f$, where $\int_S f(x') \, q(d x' \!\mid x, u) = \infty$, so that $J_2 \not\in A_0(S)$. Then $0 = f - f$, but 
$$ 0 = \int_S 0 \cdot q(d x' \!\mid x, u) \not= \int_S f(x')  \, q(d x' \!\mid x, u) + \int_S - f(x') \, q(d x' \!\mid x, u) = + \infty - \infty = + \infty.$$
The definition of $A_0(S)$ rules out such cases.}
$$  \int_S J(x') \, q(d x' \!\mid x, u) = \int_S J_1(x')  \, q(d x' \!\mid x, u) + \int_S J_2(x') \, q(d x' \!\mid x, u). $$

The set $A_0(S)$ contains all real-valued constant functions and all nonnegative lower semi-analytic functions. 
Under the (GC) condition~(\ref{eq-tmodel}), $A_0(S)$ also contains the following set of functions:
\begin{equation} \label{eq-subsetA0}
    \{ J \in A(S) \, \mid \,  J \geq f \}, \qquad \text{where} \ f(x) = - \sup_{\pi \in \Pi} J_\pi^-(x), \ \ x \in S, 
\end{equation}    
because the nonpositive function $f \in A(S)$ satisfies that $T_0(f) \geq T_-(f) = f > -\infty$ under the condition~(\ref{eq-tmodel}).
\footnote{We have $f \in A(S)$ and $T_-(f) = f$ because $f$ is the optimal cost function of the nonpositive total cost problem with one-stage cost function $-g_-$.}
In particular, $J^*, J_\infty \in A_0(S)$.

It can be verified directly that the set $A_0(S)$ is closed under addition and positive scaling: if $J_1, J_2 \in A_0(S)$, then $c_1 J_1 + c_2 J_2 \in A_0(S)$ for any scalars $c_1, c_2 \geq 0$. It can also be verified that the set $A_0(S)$ is closed under the mappings $T$ and $T_+, T_-, T_0$ defined above, under the (GC) condition~(\ref{eq-tmodel}).
\footnote{Denote $\lceil J \rceil^0 = \min \{ J, 0\}$ for any function $J$. The set $A_0(S)$ consists of those functions $J \in A(S)$ such that all the functions $T^n_0(\lceil J \rceil^0), n = 0, 1, \ldots$, lie above some function $f_J$ which does not take the value $-\infty$.  This set $A_0(S)$ is by definition closed under the mapping $T_0$, and hence, since $T_+(J) \geq T_0(J)$ for any function $J$, $A_0(S)$ is also closed under $T_+$. Similarly, since $T(J) \geq T_-(J)$, if $A_0(S)$ is closed under $T_-$, then it must also be closed under $T$. To show that $A_0(S)$ is closed under $T_-$, we use the (GC) condition~(\ref{eq-tmodel}) as follows. 
Let $f \leq 0$ be the function given in Eq.~(\ref{eq-subsetA0}). We have $f > - \infty$ under the condition~(\ref{eq-tmodel}) and $T_0(f) \geq T_-(f) = f$. 
Then for $J \in A_0(S)$, by a direct calculation, we have 
$$ T_-(J) \geq f+ T_0\big(\lceil J \rceil^0 \big) > - \infty, \qquad T^n_0\big(\lceil T_-(J) \rceil^0 \big) \geq  f + T_0^{n+1}\big(\lceil J \rceil^0 \big) \geq f + f_J > - \infty , \quad \forall \, n \geq 1,$$ 
so $T_-(J) \in A_0(S)$. This shows that the set $A_0(S)$ is closed under the mappings $T_-$ and $T$.}

For functions in $A_0(S)$, the inequalities given in the following lemma can be verified by a direct calculation, using the definitions of the mappings $T, T_+, T_-, T_0$ and the monotonicity of $T$. The inequality (\ref{eq-ud-ineq2b}) below shows a concavity property of $T$.

\begin{lem} \label{lma-ineqs}
Let $J_1, J_2 \in A_0(S)$. Then
\begin{equation}
  T^n(J_1 + J_2) \geq T_+^n(J_1) + T_-^n(J_2), \qquad T^n(J_1 + J_2) \geq T^n(J_1) + T_0^n(J_2), \qquad n \geq 1, \label{eq-ud-ineq1b}
\end{equation}  
and for any $\alpha \in [0,1]$,
\begin{equation}
T^n\big( \alpha J_1 + ( 1 - \alpha) J_2 \big) \geq \alpha \, T^n(J_1) + (1 - \alpha) \, T^n(J_2), \qquad n \geq 1. \label{eq-ud-ineq2b}
\end{equation} 
\end{lem}

\subsubsection{Convergence Results}

We give two convergence theorems below. 
They concern the convergence of $T^n(\0)(x)$ to $J^*(x)$ for \emph{some} states $x$. We will discuss their connection to Whittle's bridging condition and also derive several corollaries in the subsequent Section~\ref{sec4.2-discussion}, and we will then use examples to illustrate these theorems in Section~\ref{sec5}.

\begin{thm}[Partial convergence of value iteration] \label{thm-vi-partial}
Suppose that
\begin{equation}
   T^{\bar n}(\0)  \geq \alpha J^* + \phi \ \ \  \text{for some } \phi \in A_0(S), \ \bar n \geq 0,  \ \alpha \in (0,1].  \label{cond-thm2}
\end{equation}   
Then 
$$ \lim_{n \to \infty} T^n(\0)(x) = J^*(x), \quad  \forall \, x \in S_{0^+}: = \Big\{ x \in S \, \Big| \,  \limsup_{n \to \infty} \, \sup_{\pi \in \Pi} \E^\pi_x \{ - \phi(x_n) \} \leq 0 \Big\}.$$
In addition, if $\bar n = 0$, then 
$$ \limsup_{n \to \infty} \, T^n(\0)(x) = J^*(x), \quad \forall \, x \in S_0 : = \Big\{ x \in S \, \Big| \,  \liminf_{n \to \infty} \, \sup_{\pi \in \Pi} \E^\pi_x \{ - \phi(x_n) \} \leq 0 \Big\}; $$
if $\bar n \not=0$, then for every $x \in S_0$ such that $T^n(\0)(x)$ converges, $\lim_{n \to \infty} \, T^n(\0)(x) = J^*(x)$.
\end{thm}
\mysmallskip

\begin{thm}[Partial convergence of value iteration] \label{thm-vi-partial2}
Suppose that
\begin{equation}
   J_\infty = T (J_\infty) \quad \text{and} \quad J_\infty  \geq \alpha J^* + \phi \ \ \  \text{for some } \phi \in A_0(S),  \ \alpha \in (0,1].  \label{cond-thm3}
\end{equation}   
Then $J_\infty(x) = J^*(x)$ for all $x \in S_{0^+}$ and for all $x \in S_0$ such that $T^n(\0)(x)$ converges, where the sets $S_{0^+}, S_0$ are as given in Theorem~\ref{thm-vi-partial}.
\end{thm}

In the above theorems we have expressed the conditions in terms of the optimal cost function $J^*$ itself.
When applying the theorems, we can always relax the conditions and replace $J^*$ with the cost function of some policy, so that the conditions can be verified in practice (cf.\ the discussion in Section~\ref{sec4.2-discussion}). The function $\phi$ appearing in the conditions above reflects the difference between the optimal cost function and the iterate $T^{\bar n}(\0)$ or the limit function $J_\infty$ from value iteration. Both theorems use $\phi$ to infer the convergence of value iteration, in a manner that resembles some other sufficient conditions in the literature, which are based on the ``tail behavior'' of policies (cf.\ the subsequent Prop.~\ref{prp-vanH} and the related discussion in Section~\ref{sec4.2-discussion}).

The proofs of Theorems~\ref{thm-vi-partial} and~\ref{thm-vi-partial2} are similar. First, note that for each $x \in S$, the definition of $T_0$ implies
$T^n_0(\phi)(x) = - \sup_{\pi \in \Pi} E^{\pi}_x \{ -\phi(x_n) \},$
so the sets $S_{0^+}, S_0$ are equivalently given by
\begin{equation} \label{eq-thm2-prf0}
  S_{0^+} = \Big\{ x \in S \, \Big| \,  \liminf_{n \to \infty} \, T^n_0(\phi)(x) \geq 0 \Big\}, \qquad
    S_{0} = \Big\{ x \in S \, \Big| \,  \limsup_{n \to \infty} \, T^n_0(\phi)(x) \geq 0 \Big\}.  
\end{equation}

\begin{proof}[Proof of Theorem~\ref{thm-vi-partial}]
For any $m \geq 0$, $T^{m}\big(T^{\bar n}(\0) \big)  \geq T^m(\alpha J^* + \phi)$ by the assumption (\ref{cond-thm2}) and the monotonicity of $T$.
Since $\phi \in A_0(S)$ by assumption and $\alpha J^* \in A_0(S)$ under the (GC) condition~(\ref{eq-tmodel}), 
applying the second inequality in (\ref{eq-ud-ineq1b}) with $J_1=\alpha J^*, J_2=\phi$, we have 
\begin{align}
     T^{m+\bar n}(\0) &  \geq T^m(\alpha J^*) + T^m_0(\phi) \notag \\
        & \geq \alpha J^* + (1- \alpha) \, T^m(\0) + T^m_0(\phi), \label{eq-thm2-prf1}
\end{align}        
where in the last inequality we applied the inequality (\ref{eq-ud-ineq2b}) with $J_1=J^*, J_2=\0$ and used the fact $T^m(J^*) = J^*$.

Consider the inequality (\ref{eq-thm2-prf1}) for each $x \in S$ separately. 
Note that under the condition~(\ref{eq-tmodel}), we have $J^*(x) > - \infty$ and $\inf_{m \geq 1} T^m(\0)(x) \geq f(x) > -\infty$ where the function $f$ is as given in Eq.~(\ref{eq-subsetA0}). We also have, by the assumption $\phi \in A_0(S)$, that $\inf_{m \geq 1} T^m_0(\phi)(x) > - \infty$. 
Hence, from the inequality (\ref{eq-thm2-prf1}) for each state $x$, by taking liminf of both sides, we obtain that for each $x \in S$,
\begin{align*}
  \liminf_{n \to \infty} \, T^n(\0)(x) & \geq \alpha J^*(x) + (1 - \alpha) \liminf_{n \to \infty} \, T^n(\0)(x) + \liminf_{n \to \infty} \, T^n_0(\phi)(x).
\end{align*}  
Since $\liminf_{n \to \infty} T^n(\0)(x) > - \infty$ and $\alpha \in (0,1]$, this implies
\begin{equation}
  \liminf_{n \to \infty} \, T^n(\0)(x) \geq J^*(x) + \tfrac{1}{\alpha} \, \liminf_{n \to \infty} \, T^n_0(\phi)(x). \label{eq-thm-prf2}
\end{equation}   
For $x \in S_{0^+}$, Eqs.~(\ref{eq-thm2-prf0}), (\ref{eq-thm-prf2}) together imply
$$ \liminf_{n \to \infty} \, T^n(\0)(x) \geq J^*(x), \ \ \forall \, x \in S_{0^+}; $$ 
but $\limsup_{n \to \infty} T^n(\0)(x) \leq J^*(x)$ always. Hence $\lim_{n \to \infty}  T^n(\0)(x) = J^*(x)$ for all $x \in S_{0^+}$.

The proof of the statements for $x \in S_0$ is similar. For each $x \in S$, take a subsequence $\{n_k\}$ of integers such that $\lim_{k \to \infty} T^{n_k}_0(\phi)(x) = \limsup_{n \to \infty} T^{n}_0(\phi)(x)$. We have, by inequality (\ref{eq-thm2-prf1}),
$$ T^{n_k +\bar n}(\0)(x) \geq \alpha J^*(x) + (1- \alpha) \, T^{n_k}(\0)(x) + T^{n_k}_0(\phi)(x), \qquad k \geq 1.$$
As we showed earlier, $J^*(x) > -\infty$, $\inf_{m \geq 1} T^{m}(\0)(x) > - \infty$, and $\inf_{m \geq 1} T^{m}_0(\phi)(x) >- \infty$. 
So, by taking limsup of both sides of the preceding inequality, we obtain
\begin{equation} \label{eq-thm2-prf3}
  \limsup_{k \to \infty} \,T^{n_k +\bar n}(\0)(x) \geq \alpha J^*(x) + (1 - \alpha) \, \limsup_{k \to \infty} \, T^{n_k}(\0)(x)  + \limsup_{n \to \infty} \, T^{n}_0(\phi)(x).
\end{equation}  
If $\bar n = 0$ and $x \in S_0$, then Eqs.~(\ref{eq-thm2-prf0}), (\ref{eq-thm2-prf3}) together imply
$$ \limsup_{k \to \infty} \, T^{n_k}(\0)(x) \geq \alpha J^*(x) + (1 - \alpha) \, \limsup_{k \to \infty} \, T^{n_k}(\0)(x),$$
and since $\limsup_{k \to \infty} \, T^{n_k}(\0)(x) > - \infty$ and $\alpha \in (0,1]$, we then have
$$\limsup_{k \to \infty} \, T^{n_k}(\0)(x) \geq J^*(x) \quad \Longrightarrow \quad \limsup_{n \to \infty} \, T^{n}(\0)(x) \geq J^*(x), $$
which is equivalent to $\limsup_{n \to \infty} T^{n}(\0)(x) = J^*(x)$ [cf.\ Eq.~(\ref{eq-ineq0})].
If $\bar n \not= 0$ but $x \in S_0$ and $T^n(\0)(x)$ converges, then the inequality~(\ref{eq-thm2-prf3}) together with Eq.~(\ref{eq-thm2-prf0}) implies that
$$  \lim_{n \to \infty} \, T^{n}(\0)(x) \geq \alpha J^*(x) + (1 - \alpha) \lim_{n \to \infty} \, T^{n}(\0)(x),$$
which in turn implies $\lim_{n \to \infty} T^{n}(\0)(x) = J^*(x)$, similar to the preceding proof.
\end{proof}

\begin{proof}[Proof of Theorem~\ref{thm-vi-partial2}]
Recall that $J_\infty \leq J^*$, so to show $J_\infty(x) = J^*(x)$ for some state $x$, it is sufficient to show $J_\infty(x) \geq J^*(x)$. 

Using the condition~(\ref{cond-thm3}): $J_\infty \geq \alpha J^* + \phi$ where $\phi \in A_0(S)$, and using also Lemma~\ref{lma-ineqs}, we have, similar to the proof of Theorem~\ref{thm-vi-partial}, that for any $m \geq 0$,
\begin{equation} \label{eq-thm3-prf1}
   J_\infty = T^m(J_\infty) \geq T^m( \alpha J^* + \phi) \geq \alpha J^* + (1 - \alpha) \, T^m(\0) + T_0^m(\phi).
\end{equation}   
Consider this inequality for each $x \in S$ separately. For each $x$, take a subsequence $\{n_k\}$ of integers such that 
$\lim_{k \to \infty} T^{n_k}(\0)(x) = \limsup_{n \to \infty} \, T^n(\0)(x)$. By the inequality (\ref{eq-thm3-prf1}), we have
\begin{equation} \label{eq-thm3-prf2}
   J_\infty(x)   \geq \alpha J^*(x) +  (1 - \alpha) \,T^{n_k}(\0)(x) + T^{n_k}_0(\phi)(x), \qquad \forall \, k \geq 1.
\end{equation}   
As discussed in the proof of Theorem~\ref{thm-vi-partial}, we have $J^*(x) > - \infty$, $\inf_{m \geq 1} T^{m}(\0)(x) > - \infty$, and $\inf_{m \geq 1} T^{m}_0(\phi)(x) > - \infty$. Consequently, by letting $k$ go to $\infty$ in the preceding inequality, we obtain
\begin{align}
  J_\infty(x)  & \geq \alpha J^*(x) + (1 - \alpha) \,\limsup_{n \to \infty} \, T^n(\0)(x) + \liminf_{n \to \infty} \, T^n_0(\phi)(x)  \label{eq-thm3-prf3} \\
& \geq \alpha J^*(x) + (1 - \alpha)  J_\infty(x) + \liminf_{n \to \infty} \, T^n_0(\phi)(x), \notag
\end{align}
and since $J_\infty(x) > - \infty$ and $\alpha \in (0,1]$, this implies
$$ J_\infty(x) \geq J^*(x) + \tfrac{1}{\alpha} \, \liminf_{n \to \infty} \, T^n_0(\phi)(x).$$
If $x \in S_{0^+}$, the preceding inequality together with Eq.~(\ref{eq-thm2-prf0}) yields $J_\infty(x) \geq J^*(x)$, which implies $J_\infty(x) = J^*(x)$, as discussed earlier. This proves the statement for $x \in S_{0^+}$.

From the inequality (\ref{eq-thm3-prf1}), we also have that for any $x \in S$,
\begin{equation} \label{eq-thm3-prf4}
    J_\infty(x)  \geq \alpha J^*(x) + (1 - \alpha) \liminf_{n \to \infty} \, T^n(\0)(x) + \limsup_{n \to \infty} \, T^n_0(\phi)(x).
\end{equation}    
This inequality is obtained similarly to the inequality (\ref{eq-thm3-prf3}), by starting from the inequality (\ref{eq-thm3-prf2}) but with the subsequence $\{n_k\}$ being such that $\lim_{k \to \infty} T^{n_k}_0(\phi)(x) = \limsup_{n \to \infty} \, T^n_0(\phi)(x)$.
If $x \in S_0$ and $T^n(\0)(x)$ converges so that $J_\infty(x) = \lim_{n \to \infty} T^n(\0)(x)$, then the inequality (\ref{eq-thm3-prf4}) together with Eq.~(\ref{eq-thm2-prf0}) implies
$ J_\infty(x)  \geq \alpha J^*(x) + (1 - \alpha) J_\infty(x)$,
which is equivalent to $J_\infty(x) \geq J^*(x)$ (since $J_\infty(x) > - \infty$ and $\alpha \in (0,1]$), and this is in turn equivalent to $J_\infty(x) = J^*(x)$, as discussed earlier. This completes the proof.
\end{proof}

\subsubsection{Implications and Discussion} \label{sec4.2-discussion}

We start with two direct consequences of Theorem~\ref{thm-vi-partial}. The first corollary simply relaxes the condition of Theorem~\ref{thm-vi-partial} using the fact $J^* \leq J_\pi$ for any policy $\pi$; the relaxed condition is easier to verify in practice.

\begin{cor} \label{cor-vi-partial0}
Suppose there exists a policy $\pi$ satisfying $T^{\bar n}(\0)  \geq \alpha J_\pi + \phi$ for some $\phi \in A_0(S)$, $\bar n \geq 0$, and $\alpha \in (0,1]$. Then the conclusions of Theorem~\ref{thm-vi-partial} hold.
\end{cor}

To prepare for a discussion on the relation of our results with the bridging condition, let us state another evident implication of Theorem~\ref{thm-vi-partial} by letting $\phi(\cdot)\equiv 0$ in the theorem.

\begin{cor} \label{cor-vi-partial1}
We have  $T^n(\0) \to J^*$ if there exist $\bar n \geq 0$ and $\alpha \in (0,1]$ such that 
\begin{equation}
  T^{\bar n}(\0)  \geq \alpha J^* \quad \text{or} \ \ \  T^{\bar n}(\0)  \geq \alpha J_\pi \quad \text{for some policy} \  \pi. \label{cond-alt}
\end{equation}
\end{cor}

\begin{rem} \label{rmk-bridging}
We can compare the condition of Cor.~\ref{cor-vi-partial0} with the bridging condition for the positive model (P)~\cite{Whit79}, which can be stated in our notation as $T^{\bar n}(\0)  \geq \alpha J_\pi$ for some policy $\pi$, i.e., the condition (\ref{cond-alt}) in Cor.~\ref{cor-vi-partial1}.
According to our understanding of~\cite{Whit79}, the name ``bridging'' reflects the fact that the condition relates the finite-horizon cost $T^{\bar n}(\0)$ to the infinite-horizon cost $J_\pi$. In this sense, the conditions of our theorems are also of the ``bridging-type,'' and in our theorems, the differences between the finite-horizon and infinite-horizon costs are reflected in the function $\phi$ and used to infer the convergence of value iteration at various states. 
For the case $g \geq 0$, 
Whittle~\cite{Whit79} and then Hartley~\cite{Har80} gave two different proofs, establishing that under the bridging condition, $T^n(J) \to J^*$ for all functions $J$ with $0 \leq J \leq c J^*$ for some real $c$. Our proofs of Theorems~\ref{thm-vi-partial},~\ref{thm-vi-partial2} indeed started out similarly to Hartley's.

For the (GC) model, we see that under the condition~(\ref{cond-alt}),
the desired convergence $T^n(\0) \to J^*$ holds. Thus it may seem that we could have taken the original bridging condition~(\ref{cond-alt}), which is much simpler than the conditions of Theorem~\ref{thm-vi-partial} or Cor.~\ref{cor-vi-partial0}, for the more general total cost problems in (GC). But the fact is that condition (\ref{cond-alt}) can be not only restrictive but also unnatural for the (GC) model, where $J^*$ can take negative values. In particular, for those states $x$ with $J^*(x) \leq 0$, $\alpha J^*(x) \geq J^*(x)$ for all $\alpha \in (0,1]$, whereas it is possible that $T^n(\0)(x) < J^*(x)$ for all $n$ and then the inequalities in (\ref{cond-alt}) cannot be satisfied. This is the reason for introducing the function $\phi$ in formulating the conditions of Theorems~\ref{thm-vi-partial},~\ref{thm-vi-partial2}.

If the conditions of Theorem~\ref{thm-vi-partial} are met with $S_{0^+} = S$, we have $T^n(\0) \to J^*$, and this combined with Theorem~\ref{thm-vi} then implies, by the monotonicity of $T$, that $T^n(J) \to J^*$ for all lower semi-analytic functions $J$ such that $0 \leq J \leq c (J^{*+} + J^{*-})$ for some real $c$. This set of $J$ then resembles the set of initial functions in the conclusion of~\cite{Whit79} for the model (P) mentioned earlier, although here $J^*$ need not lie in this set because it can take negative values.
\qed
\end{rem}

We now discuss another special case of Theorem~\ref{thm-vi-partial} and relate it to two existing sufficient conditions for the convergence of value iteration. 
Take $\bar n = 0$ in Theorem~\ref{thm-vi-partial} with $\phi = - J^*$, $\alpha = 1$. We deduce that if $J^*$ is real-valued and $- J^* \in A_0(S)$, then the conclusions given in the subsequent Prop.~\ref{prp-vanH} must hold. 
However, in this case, these conditions on $J^*$ turn out to be unnecessary (they are artifacts of the limitation of the proof technique we used), as a short direct proof given below shows. This result, Prop.~\ref{prp-vanH}, relates to a theorem of van Hee et al.~\cite[Theorem 3.5]{vanH77}: Indeed the first half of Prop.~\ref{prp-vanH} is a ``pointwise version'' of \cite[Theorem 3.5]{vanH77}, and the proof argument for Prop.~\ref{prp-vanH}, included below for completeness, is essentially from~\cite{vanH77}.

\begin{prop}[{Cf.\ van Hee et al.~\cite[Theorem 3.5]{vanH77}, van der Wal~\cite[Theorem 3.3]{vdW81}}] \label{prp-vanH}
The following holds:
$$ \limsup_{n \to \infty} \, T^n(\0)(x) = J^*(x), \ \   \forall \, x \in S_0, \qquad \quad \lim_{n \to \infty} T^n(\0)(x) = J^*(x), \ \   \forall \, x \in S_{0^+},$$
where the sets $S_0$ and $S_{0^+}$ are given by
$$ S_{0} = \Big\{ x \in S \, \Big| \,  \liminf_{n \to \infty} \, \sup_{\pi \in \Pi} \E^\pi_x \{ J^*(x_n) \} \leq 0 \Big\}, \qquad
    S_{0^+} = \Big\{ x \in S \, \Big| \,  \limsup_{n \to \infty} \, \sup_{\pi \in \Pi} \E^\pi_x \{ J^*(x_n) \} \leq 0 \Big\}.$$
In particular, if     
$\limsup_{n \to \infty} \, \sup_{\pi \in \Pi} \E^{\pi}_x \big\{ J^{*}(x_n) \big\} \leq 0$ for all $x \in S$,
then $T^n(\0) \to J^*$.
\end{prop}
\begin{proof}
Consider any policy $\bar \pi$ and state $x$. Let $J_{\bar \pi,n}(x)$ be the $n$-stage cost of $\bar \pi$ for the initial state $x$. Since $T^n(J^*) = J^*$, by the definition of $T^n$, we have
$$ J^*(x)  = T^n(J^*)(x) \leq J_{\bar \pi,n}(x) + \E^{\bar \pi}_x \big\{ J^*(x_n) \big\} \leq  J_{\bar \pi,n}(x) + \sup_{\pi \in \Pi} \E^{\pi}_x \big\{ J^*(x_n) \big\}, 
$$
so, by minimizing the right-hand side over $\bar \pi$, we have
$$ J^*(x)  \leq T^n(\0)(x) + \sup_{\pi \in \Pi} \E^{\pi}_x \big\{ J^*(x_n) \big\}.$$
By taking liminf over $n$ in the right-hand side of the preceding inequality, it follows that
\begin{align*}
  J^*(x) &  \leq \limsup_{n \to \infty} \, T^n(\0)(x) +\liminf_{n \to \infty} \, \sup_{\pi \in \Pi} \E^{\pi}_x \big\{ J^*(x_n) \big\},  \\ 
  J^*(x) & \leq \liminf_{n \to \infty} \,T^n(\0)(x) + \limsup_{n \to \infty} \, \sup_{\pi \in \Pi} \E^{\pi}_x \big\{ J^*(x_n) \big\}.
\end{align*}  
Using the definitions of the sets $S_0$, $S_{0^+}$, we then have that $ \limsup_{n \to \infty} \, T^n(\0)(x) \geq J^*(x)$ for $x \in S_0$ and $\liminf_{n \to \infty} T^n(\0)(x) \geq J^*(x)$ for $x \in S_{0^+}$. Since $ \limsup_{n \to \infty} \, T^n(\0)(x) \leq J^*(x)$ for all $x$, the proposition follows.
\end{proof}

\begin{rem} 
For the (GC) model, similar to Lemma~\ref{lma-1}, it holds
\footnote{The proof of Eq.~(\ref{eq-Js-tc}) is similar to that of Lemma~\ref{lma-1}: By Lemma~\ref{lma-B}(b), we have
$$ J_\pi(x) = \E^\pi_x \left\{ \sum_{k=0}^{n-1} g(x_k, u_k) \right\} + \E^\pi_x \left\{ \sum_{k=n}^{\infty} g(x_k, u_k) \right\} \geq  \E^\pi_x \left\{ \sum_{k=0}^{n-1} g(x_k, u_k) \right\} + \E^\pi_x \left\{ J^*(x_n) \right\}.$$
Since $+ \infty > J_\pi(x) = \lim_{n \to \infty} \E^\pi_x \left\{ \sum_{k=0}^{n-1} g(x_k,u_k) \right\} > - \infty$ by the (GC) condition~(\ref{eq-tmodel}), Eq.~(\ref{eq-Js-tc}) follows.} 
that for any state $x \in S$ and policy $\pi$ such that $J_\pi(x) < + \infty$,
\begin{equation} \label{eq-Js-tc}
  \limsup_{n \to \infty} \, \E^{\pi}_x \big\{ J^{*}(x_n) \big\} \leq 0.
\end{equation}  
We can compare Eq.~(\ref{eq-Js-tc}) with the sufficient condition for $T^n(\0)(x) \to J^*(x)$ given by the definition of $S_{0^+}$ in Prop.~\ref{prp-vanH}:
$\limsup_{n \to \infty} \, \sup_{\pi \in \Pi} \E^\pi_x \{ J^*(x_n) \}  \leq 0,$
which not only implies Eq.~(\ref{eq-Js-tc}) but also requires that as $n$ increases, the expectations $\E^{\pi}_x \big\{ J^{*}(x_n) \big\}$ under different policies $\pi$ must approach the interval $(-\infty, 0]$ at ``comparable speeds.''\qed
\end{rem}

\begin{rem}
By Lemma~\ref{lma-B}(b), $\E^\pi_x \left\{ \sum_{k=n}^\infty g(x_k, u_k) \right\} \geq \E^\pi_x \left\{ J^*(x_n) \right\}$ for any policy $\pi$, state $x$ and $n \geq 0$. So Prop.~\ref{prp-vanH} implies the following ``tail conditions'' for convergence of value iteration, which can be relatively easier to verify in practice than the conditions of Prop.~\ref{prp-vanH}:
\begin{equation}
   \limsup_{n \to \infty} \, T^n(\0)(x) = J^*(x), \qquad \forall \, x \ s.t. \  \liminf_{n \to \infty} \, \sup_{\pi \in \Pi} \E^\pi_x \left\{ \sum_{k=n}^\infty g(x_k, u_k) \right\} \leq 0, 
\end{equation}
\begin{equation} \label{eq-tail2}
   \lim_{n \to \infty} \, T^n(\0)(x) = J^*(x), \qquad \forall \, x  \ s.t. \  \limsup_{n \to \infty} \, \sup_{\pi \in \Pi} \E^\pi_x \left\{ \sum_{k=n}^\infty g(x_k, u_k) \right\} \leq 0.
\end{equation}
The statement in Eq.~(\ref{eq-tail2}) is a ``pointwise version'' of a theorem in Hern\'{a}ndez-Lerma and Lasserre~\cite[Theorem 9.3.5(b)]{HL99}.
\footnote{The statement of~\cite[Theorem 9.3.5(b)]{HL99} differs from Eq.~(\ref{eq-tail2}) here in that it has ``liminf'' instead of ``limsup'' in the condition on $x$. However, an inspection of its proof shows that the correct version should be as given in Eq.~(\ref{eq-tail2}). Indeed, otherwise, Example \ref{ex-nolimit} (Section~\ref{sec5}) would serve as a counterexample: there the condition $\liminf_{n \to \infty} \, \sup_{\pi \in \Pi} \E^\pi_x \left\{ \sum_{k=n}^\infty g(x_k, u_k) \right\} \leq 0$ is met at all states $x$, and yet $\{T^n(\0)\}$ does not have a limit.}
\qed
\end{rem}

\subsubsection*{The Special Case (UD)}

We consider now the $\beta$-discounted problems in the (UD) class (cf.\ Def.~\ref{def-ud}). For these problems, the dynamic programming operators $T, T_+, T_-, T_0$ are defined as for the (GC) model, except that the integrals in their definitions are now multiplied by the discount factor $\beta$ (cf.\ the discussion preceding Cor.~\ref{cor-ud1}). As mentioned before Def.~\ref{def-ud}, it is a standard procedure to reformulate these discounted problems as undiscounted total cost problems on an augmented state space $S \cup \{\Delta\}$ with $\Delta$ representing an absorbing, cost-free, termination state. By first applying Theorems~\ref{thm-vi-partial},~\ref{thm-vi-partial2} to the equivalent undiscounted problems, and then expressing the results in terms of the original discounted problems, we obtain the following version of Theorems~\ref{thm-vi-partial} and~\ref{thm-vi-partial2} for the unbounded discounted model (UD).

\begin{cor} \label{cor-ud-vi0}
{\rm (UD)} \ The conclusions of Theorems~\ref{thm-vi-partial},~\ref{thm-vi-partial2} hold, under the same conditions of these theorems, if the sets $S_{0^+}, S_0$ in their statements are redefined as 
\begin{align*}
S_{0^+} & : = \Big\{ x \in S \, \Big| \,  \limsup_{n \to \infty} \,  \sup_{\pi \in \Pi} \E^\pi_x \{ - \beta^n \phi(x_n) \} \leq 0 \Big\}, \\
  S_0 & : = \Big\{ x \in S \, \Big| \,  \liminf_{n \to \infty} \,  \sup_{\pi \in \Pi} \E^\pi_x \{ - \beta^n \phi(x_n) \} \leq 0 \Big\}.
\end{align*}
\end{cor}

\begin{cor} \label{cor-ud-vi}
{\rm (UD)} \ 
Suppose $\sup_{x \in S} \sup_{\pi\in\Pi}  J_\pi^-(x) < \infty$. If $T_+^{\bar n}(\0) \geq \alpha J^* + b$ for some $\bar n \geq 0$, $\alpha \in (0,1]$, and scalar $b \leq 0$ (which holds if $T_+^{\bar n}(\0) \geq \alpha J_\pi + b$ for some policy $\pi$),
then $T^n(\0) \to J^*$.
\end{cor}

\begin{proof}
Using the assumption and the first inequality in (\ref{eq-ud-ineq1b}), we have
$$ T^{\bar n}(\0) \geq T_+^{\bar n}(\0)  + T_-^{\bar n}(\0)  \geq \alpha J^* + b + T_-^{\bar n}(\0) \geq \alpha J^* + \bar b,$$
where $\bar b$ in the last expression is given by $\bar b = b  - \sup_{x \in S} \sup_{\pi \in \Pi}  J_\pi^-(x) > - \infty$, and the last inequality holds because $T_-^{\bar n}(\0)(x) \geq - \sup_{\pi \in \Pi} J_\pi^-(x)$ for all $x \in S$.
Let $\phi(\cdot) \equiv \bar b$ on $S$. Since for all states $x$ and policies $\pi$, $\E^\pi_x \{ - \beta^n \phi(x_n) \} \leq \beta^n |\bar b| \to 0$, the desired result follows from Cor.~\ref{cor-ud-vi0}. 
\end{proof}

\begin{rem}
The condition $T_+^{\bar n}(\0) \geq \alpha J^* + b$ in Cor.~\ref{cor-ud-vi} can also be compared to the bridging condition, since it relates the infinite-horizon cost $J^*$ to the finite-horizon cost $T_+^{\bar n}(\0)$, albeit $T_+^{\bar n}(\0)$ is from a different problem, the one associated with $g_+$. Note that this condition does \emph{not} require the convergence of value iteration $T^n_+(\0)$ to $J^{*+}$.\qed
\end{rem}

\section{Illustrative Examples} \label{sec5}
\markboth{\rm \S \ref{sec5}. Illustrative Examples}{\rm \S \ref{sec5}. Illustrative Examples}

We start with two examples, which are modified versions of \cite[Examples 6.6, 6.11]{Fein02}. They are counterexamples to the convergence of the value iteration $T^n(\0)$ for the (GC) model.

\begin{example}[The limit of $\{ T^n(\0) \}$ does not exist] \label{ex-nolimit}
The problem is deterministic. The state space is a subset of nonnegative integer pairs: $S = \{(0,0)\} \cup \big( \{ 1, 2, \ldots \} \times \{0, 1, 2, \ldots \} \big)$. For any state $(i,j)$ with $i > 0$, only control $0$ can be applied, under which the system moves to state $(i+1, j)$. At state $(0,0)$, the feasible control set is $\{0,1, 2, \ldots\}$, and with a feasible control $k$, the system moves to state $(1, k)$. All the one-stage costs are zero, except that 
$$g\big((2i+2,i), 0 \big) = 2, \qquad g\big( (2i+1,i), 0 \big) = g\big( (2i+3,i), 0 \big) = -1, \qquad i = 0, 1, \ldots.$$
At state $(0,0)$, the optimal cost $J^*\big( (0,0) \big) = 0$, but
$$\liminf_{n \to \infty} \, T^n(\0)\big( (0,0) \big) = -1, \qquad \limsup_{n \to \infty} \, T^n(\0)\big( (0,0) \big) = 0.$$
Thus, $\{ T^n(\0)\}$ does not have a limit, although $\limsup_{n \to \infty} \, T^n(\0) = J^*$, as can be verified. \qed
\end{example}

\begin{example}[$J_\infty \not= J^*$]
The state space, feasible control sets, and the dynamics of the system are as given in Example~\ref{ex-nolimit}. All the one-stage costs are zero, except that 
$$ g \big( (i,i), 0  \big) = 1, \qquad i = 0, 1, 2, \ldots.$$
Then at state $(0,0)$, $T^n(\0)\big((0,0) \big) = 0$ for all $n \geq 0$, but $J^*\big((0,0)\big) = 1$. \qed
\end{example}

The above examples have a countable state space. It is well-known, however, that even for finite state and control problems, $J_\infty \not= J^*$ can happen, as Example~\ref{ex-rel2P} showed. 

We now use these examples to illustrate the convergence of value iteration from above stated in Theorem~\ref{thm-vi}, and the partial convergence of value iteration stated in Theorems~\ref{thm-vi-partial},~\ref{thm-vi-partial2} or their corollaries.

\begin{example}[Example~\ref{ex-rel2P} Cont'd]
Consider Example~\ref{ex-rel2P}. It has $3$ states $\{0,1,2\}$. Let us write a cost function $J$ in vector form as $J = (J(0), J(1), J(2))$. In this problem any policy $\pi$ is optimal, with $J_\pi=J^* = (0, 1, 0)$, whereas $J_\infty = T(\0) = (0,1,-1)$ and it is a fixed point of $T$. Let us apply Theorem~\ref{thm-vi-partial} (or equivalently, Cor.~\ref{cor-vi-partial0} or Theorem~\ref{thm-vi-partial2}) with $\bar n = 1, \alpha=1$ and $\phi=T(\0) - J_\pi = (0,0, -1)$.  Starting from state $x=2$, there is a policy to keep the system forever at state $2$, so we have 
$$  \liminf_{n \to \infty} \, \sup_{\pi \in \Pi} \E^\pi_x \{ - \phi(x_n) \} = \limsup_{n \to \infty} \, \sup_{\pi \in \Pi} \E^\pi_x \{ - \phi(x_n) \} = 1, \qquad x = 2.$$
Starting from other states $x$, the above liminf and limsup expressions are equal to zero. So the sets $S_{0^+}, S_{0}$ in Theorem~\ref{thm-vi-partial} are given by $S_{0^+} = S_{0} = \{0, 1\}$, and according to Theorem~\ref{thm-vi-partial} or Cor.~\ref{cor-vi-partial0}, for $x \in \{0,1\}$, $T^n(\0)(x)$ must converge to $J^*(x)$, which is what we observe in this example.

Consider now value iteration starting from $J \geq J^*$. We have $J^{*+}=(0,1,0), J^{*-}=(0,0,0)$, so according to Theorem~\ref{thm-vi}, $T^n(J) \to J^*$ for any $J=(0,c,0)$, $c \geq 1$. Indeed we have $T(J) = (0,1,0) = J^*$ in this example.
 \qed 
\end{example}

\begin{example}[Example~\ref{ex-nolimit} Cont'd]
Consider Example~\ref{ex-nolimit}. 
In this problem any policy $\pi$ is optimal. Let us apply Theorem~\ref{thm-vi-partial} or Cor.~\ref{cor-vi-partial0} with $\bar n = 0, \alpha=1$ and the function $\phi = - J_\pi = - J^*$. In particular, we have 
$$ J^*\big( (k, i) \big) = \begin{cases}   1 &   k = 2i+2; \\  -1 & k = 2i+3; \\ 0 & \text{otherwise}; \end{cases} $$  
so
\begin{equation} \label{eq-ex4-a}
    \phi\big( (k, i) \big) = 0 \ \ \text{if} \  k \geq 2i+4; \qquad  \phi\big( (k, i) \big) < 0 \ \ \ \text{if and only if} \  k = 2 i +2.
\end{equation}    
Starting from any state $x=(k, i) \not= (0,0)$, the states evolve as $(k+1, i), (k+2, i), \ldots$, so 
$$ \liminf_{n \to \infty} \, \sup_{\pi \in \Pi} \E^\pi_x \{ - \phi(x_n) \} = \limsup_{n \to \infty} \, \sup_{\pi \in \Pi} \E^\pi_x \{ - \phi(x_n) \} =  0, \qquad x \not=(0,0).$$
For state $x=(0,0)$, we have $\limsup_{n \to \infty} \, \sup_{\pi \in \Pi} \E^\pi_x \{ - \phi(x_n) \}  = 1$ because for any positive even number $n$, there exists a policy $\pi$ with $\E^\pi_x \{ - \phi(x_n) \}  = \E^\pi_x \{ J^*(x_n) \}  = 1$. On the other hand, in view of the second relation in Eq.~(\ref{eq-ex4-a}), under any policy $\pi$, we have $\E^\pi_x \{ - \phi(x_n) \} \leq 0$ if $n$ is an odd number, so
$$ \liminf_{n \to \infty} \, \sup_{\pi \in \Pi} \E^\pi_x \{ - \phi(x_n) \}  = 0, \qquad x = (0,0).$$
Thus, the sets $S_{0^+}, S_0$ in Theorem~\ref{thm-vi-partial} are given by
$$S_{0^+} = S \setminus \{(0,0)\}, \qquad S_0 = S.$$
Since $S_0=S$, according to Theorem~\ref{thm-vi-partial} or Cor.~\ref{cor-vi-partial0}, $\limsup_{n \to \infty} T^n(\0) = J^*$, which is consistent with what actually happens in this example. According to Theorem~\ref{thm-vi-partial} or Cor.~\ref{cor-vi-partial0}, $T^n(\0)(x) \to J^*(x)$ for all $x \in S_{0^+}$, i.e., for all $x \not= (0,0)$, which is also consistent with what we see in the example.

Consider now the convergence of value iteration from above. Let us calculate $J^{*+}, J^{*-}$:
$$ J^{*+}\big( (k, i) \big) = \begin{cases}   2 &   k \leq 2i+2; \\ 0 & \text{otherwise}; \end{cases} \qquad \quad
      J^{*-}\big( (k, i) \big) = \begin{cases}   2 &   k \leq 2i+1; \\  1 & k = 2i+2 \ \text{or} \ 2i+3; \\ 0 & \text{otherwise}. \end{cases} $$
Comparing $J^{*+} + J^{*-}$ with the optimal cost function $J^*$ calculated earlier, we see that
according to Theorem~\ref{thm-vi}, if we start value iteration, for instance, with the function $J$ given by
$$    J\big( (k, i) \big) = \begin{cases}  c & k \leq 2i+2; \\ -1 &  k = 2i+3; \\  0 & \text{otherwise}; \end{cases}  \qquad \text{where} \ \ c \geq 1,$$
then we should have $T^n(J) \to J^*$. Indeed this is the case; in particular, for state $(0,0)$, we have that $T^n(J)\big((0,0)\big) = 0 = J^*\big( (0,0)\big)$ for all $n \geq 3$. 

As discussed immediately after Theorem~\ref{thm-vi}, we can also start value iteration from $J^{*+} - J^{*-}$ to have monotonic convergence: $T^n(J^{*+} - J^{*-}) \downarrow J^*$. It can be seen that this is the case here; actually, $J^* = J^{*+} - J^{*-}$ in this example.
\qed
\end{example}

\section{Concluding Remarks} \label{sec6}
\markboth{\rm \S \ref{sec6}. Concluding Remarks}{\rm \S \ref{sec6}. Concluding Remarks}

In this paper we considered a general class of total cost problems under the (GC) condition, and we presented several convergence results that characterize the behavior of value iteration for this class. 
Our result on convergence of value iteration from above holds for all problems in the (GC) class. It is also relevant to the policy iteration method, since policy iteration works with functions lying above the optimal cost function, and it sheds some light on why policy iteration can get stuck at a suboptimal policy in the total cost case (cf.~\cite[Example 7.3.4]{puterman94}). In the context of finding near optimal policies, another subject to which this result may be potentially relevant is finding $\epsilon \ell$-optimal stationary policies, where $\ell$ is a given nonnegative function on $S$ (see e.g., \cite{Fein87b, FeS83,Orn69,ScS87};
see also the book \cite[Chap.\ 7.2.3]{puterman94}, the survey~\cite[Chap.\ 6.7-6.9]{Fein02} and the references therein); 
investigation of this topic is needed.
Regarding value iteration, a question that requires further research is how to find efficiently an initial function in the set with guaranteed convergence. We can start with the function $J^{*+}$ or $J^{*+} - J^{*-}$, for example, but calculating $J^{*+}$ or $J^{*-}$ can be as difficult when the control space is infinite. We can also start with the cost function of some policy, but in general it can be hard to verify whether this function satisfies the convergence condition of our theorem. Another interesting question is whether by choosing suitably an initial function, one can not only ensure but also speed up the convergence of value iteration. (For average cost problems, Chen and Meyn~\cite{chm-vi} and Meyn~\cite{Mey08} have shown that this is the case.) 

We envision that our result on partial convergence of value iteration can be useful in practice in several ways. First, we may combine this result with the knowledge of a given problem to ensure that $T^n(\0) \to J^*$ holds in the given problem, or to find a subset of states for which the convergence holds, thereby extracting out useful information from the iterates produced by value iteration. Second, if we know a subset $C$ of states $x$ on which $\lim_{n \to \infty} T^n(\0)(x)$ or $\limsup_{n \to \infty} T^n(\0)(x)$ equals $J^*(x)$, we may estimate these limit values and then restart value iteration by setting the values of the initial function on $C$ according to the estimates. This heuristic way of applying value iteration is in the spirit of transfinite value iteration, and may turn out to have some interesting properties upon further study.

In summary, the analytical results presented in this paper can provide useful guidance or background, we believe, for the application of the value iteration method in practice, as well as for further theoretical or computational study of this method. As part of the theoretical foundations, these results can also be useful for applying learning and simulation-based approaches \cite{BET,Mey08,SUB,rl-book2} and stochastic approximation techniques~\cite{Bor08,KuY03} to solve total cost MDP. In that context, a particular subject for future research is Q-learning \cite{tsi94,wat89} and related asynchronous stochastic value iteration algorithms, and their convergence properties for total cost MDP under the (GC) condition.

\section*{Acknowledgments}
Several ideas presented in this paper grew out of a previous work \cite{YuB-mvipi} that I did with Professor Dimitri Bertsekas, at the Laboratory for Information and Decision Systems (LIDS), MIT. I thank Professor Bertsekas for helpful discussions then and during the course of writing this paper. I also thank Professors Eugene Feinberg and William Sudderth for helpful comments; Professor Sean Meyn for pointing out the material on value iteration in his book~\cite{Mey08}, which motivated the first convergence result given in this paper; and two anonymous reviewers, whose feedback helped to improve the presentation of this paper. This work was completed after I joined the Reinforcement Learning Artificial Intelligence (RLAI) research group at the Department of Computing Science, University of Alberta, and I am very thankful for the support from Alberta Innovates -- Technology Futures and from the RLAI group for my research.

\addcontentsline{toc}{section}{References} 
\markboth{\rm References}{\rm References}
\bibliographystyle{plain} 
\let\oldbibliography\thebibliography
\renewcommand{\thebibliography}[1]{%
  \oldbibliography{#1}%
  \setlength{\itemsep}{0pt}%
}
{\fontsize{9}{11} \selectfont
\bibliography{totalcost_bib}}

\begin{thebibliography}{10}

\bibitem{b77}
D.~P. Bertsekas.
\newblock Monotone mappings with application in dynamic programming.
\newblock {\em SIAM J. Control Opt.}, 15:438--464, 1977.

\bibitem{bs}
D.~P. Bertsekas and S.~E. Shreve.
\newblock {\em Stochastic Optimal Control: The Discrete Time Case}.
\newblock Academic Press, New York, 1978.

\bibitem{BeT-ssp}
D.~P. Bertsekas and J.~N. Tsitsiklis.
\newblock An analysis of stochastic shortest path problems.
\newblock {\em Math. Oper. Res.}, 16:580--595, 1991.

\bibitem{BET}
D.~P. Bertsekas and J.~N. Tsitsiklis.
\newblock {\em Neuro-Dynamic Programming}.
\newblock Athena Scientific, Belmont, 1996.

\bibitem{Blk-positive}
D.~Blackwell.
\newblock Positive dynamic programming.
\newblock In {\em Proc. 5th Berkeley Sympos. Math. Satist. and Probability},
  pages 415--418, 1965.

\bibitem{Blk-borel}
D.~Blackwell.
\newblock A {Borel} set not containing a graph.
\newblock {\em Ann. Math. Statist.}, 39:1345--1347, 1968.

\bibitem{BFO74}
D.~Blackwell, D.~Freedman, and M.~Orkin.
\newblock The optimal reward operator in dynamic programming.
\newblock {\em Ann. Probability}, 2:926--941, 1974.

\bibitem{Bor08}
V.~S. Borkar.
\newblock {\em Stochastic Approximation: A Dynamic Viewpoint}.
\newblock Hindustan Book Agency, New Delhi, 2008.

\bibitem{chm-vi}
R.~R. Chen and S.~Meyn.
\newblock Value iteration and optimization of multiclass queueing networks.
\newblock {\em Queueing Systems}, 32:65--97, 1999.

\bibitem{Dud02}
R.~M. Dudley.
\newblock {\em Real Analysis and Probability}.
\newblock Cambridge University Press, Cambridge, 2002.

\bibitem{DuHP12}
F.~Dufour, M.~Horiguchi, and A.~B. Piunovskiy.
\newblock The expected total cost criterion for {Markov} decision processes
  under constraints: a convex analytic approach.
\newblock {\em Adv. Appl. Prob.}, 44:774--793, 2012.

\bibitem{DuP13}
F.~Dufour and A.~B. Piunovskiy.
\newblock The expected total cost criterion for {Markov} decision processes
  under constraints.
\newblock {\em Adv. Appl. Prob.}, 45:837--859, 2013.

\bibitem{Dynkin79}
E.~B. Dynkin and A.~A. Yushkevich.
\newblock {\em Controlled {Markov} Processes}.
\newblock Springer, New York, 1979.

\bibitem{Fein87b}
E.~A. Feinberg.
\newblock Sufficient classes of strategies in discrete dynamic programming
  {II}: locally stationary strategies.
\newblock {\em SIAM Theory Prob. Appl.}, 32:478--493, 1987.

\bibitem{Fein02}
E.~A. Feinberg.
\newblock Total reward criteria.
\newblock In E.~A. Feinberg and A.~Shwartz, editors, {\em Handbook of {Markov}
  Decision Processes}. Springer, New York, 2002.

\bibitem{FKZ12}
E.~A. Feinberg, P.~O. Kasyanov, and N.~V. Zadoianchuk.
\newblock Average cost {Markov} decision processes with weakly continuous
  transition probabilities.
\newblock {\em Math. Oper. Res.}, 37:591--607, 2012.

\bibitem{FKZ14}
E.~A. Feinberg, P.~O. Kasyanov, and M.~Z. Zgurovsky.
\newblock Convergence of value iterations for total-cost {MDPs} and {POMDPs}
  with general state and action sets.
\newblock In {\em Proc. IEEE International Symposium on Adaptive Dynamic
  Programming and Reinforcement Learning (ADPRL)}, 2014.

\bibitem{FeS83}
E.~A. Feinberg and I.~M. Sonin.
\newblock Stationary and {Markov} policies in countable state dynamic
  programming.
\newblock In {\em Lecture Notes in Math. 1021}, pages 111--129.
  Springer-Verlag, Berlin and New York, 1983.

\bibitem{Har80}
R.~Hartley.
\newblock A simple proof of {Whittle's} bridging condition in dynamic
  programming.
\newblock {\em J. Appl. Prob.}, 17:1114--1116, 1980.

\bibitem{HL99}
O.~Hern\'{a}ndez-Lerma and J.~B. Lasserre.
\newblock {\em Further Topics on Discrete-Time {Markov} Control Processes}.
\newblock Springer, New York, 1999.

\bibitem{KreP77}
D.~M. Kreps and E.~L. Porteus.
\newblock On the optimality of structured policies in countable stage decision
  processes. {II}: positive and negative problems.
\newblock {\em SIAM J. Appl. Math.}, 32:457--466, 1977.

\bibitem{Kur-topology}
K.~Kuratowski.
\newblock {\em Topology I}.
\newblock Academic Press, New York, 1966.

\bibitem{KuY03}
H.~J. Kushner and G.~G. Yin.
\newblock {\em Stochastic Approximation and Recursive Algorithms and
  Applications}.
\newblock Springer-Verlag, New York, 2nd edition, 2003.

\bibitem{MS92}
A.~Maitra and W.~Sudderth.
\newblock The optimal reward operator in negative dynamic programming.
\newblock {\em Math. Oper. Res.}, 17:921--931, 1992.

\bibitem{MS96}
A.~Maitra and W.~Sudderth.
\newblock {\em Discrete Gambling and Stochastic Games}.
\newblock Springer, New York, 1996.

\bibitem{Mey08}
S.~Meyn.
\newblock {\em Control Techniques for Complex Networks}.
\newblock Cambridge University Press, Cambridge, 2008.

\bibitem{Orn69}
D.~Ornstein.
\newblock On the existence of stationary optimal strategies.
\newblock {\em Proc. Am. Math. Soc.}, 20:563--569, 1969.

\bibitem{puterman94}
M.~L. Puterman.
\newblock {\em Markov decision processes: {Discrete} stochastic dynamic
  programming}.
\newblock John Wiley \& Sons, New York, 1994.

\bibitem{Schal75}
M.~Sch{$\ddot{\mathrm{a}}$}l.
\newblock Conditions for optimality in dynamic programming and for the limit of
  $n$-stage optimal policies to be optimal.
\newblock {\em Z. Wahrscheinlichkeitstheorie verw. Gebiete}, 32:179--196, 1975.

\bibitem{Sch83}
M.~Sch{$\ddot{\mathrm{a}}$}l.
\newblock Stationary policies in dynamic programming models under compactness
  assumptions.
\newblock {\em Math. Oper. Res.}, 8:366--372, 1983.

\bibitem{ScS87}
M.~Sch{$\ddot{\mathrm{a}}$}l and W.~Sudderth.
\newblock Stationary policies and {Markov} policies in {Borel} dynamic
  programming.
\newblock {\em Probab. Th. Rel. Fields}, 74:91--111, 1987.

\bibitem{Shr79}
S.~E. Shreve.
\newblock Resolution of measurability problems in discrete-time stochastic
  control.
\newblock In {\em Stochastic Control Theory and Stochastic Differential
  Systems}, pages 580--587. Springer, Berlin, 1979.

\bibitem{ShrB79}
S.~E. Shreve and D.~P. Bertsekas.
\newblock Universally measurable policies in dynamic programming.
\newblock {\em Math. Oper. Res.}, 4:15--30, 1979.

\bibitem{SUB}
R.~S. Sutton and A.~G. Barto.
\newblock {\em Reinforcement Learning}.
\newblock MIT Press, Cambridge, 1998.

\bibitem{rl-book2}
C.~Szepesv\'{a}ri.
\newblock {\em Algorithms for Reinforcement Learning}.
\newblock Morgan \& Claypool, 2010.

\bibitem{tsi94}
J.~N. Tsitsiklis.
\newblock Asynchronous stochastic approximation and {Q}-learning.
\newblock {\em Mach. Learn.}, 16:185--202, 1994.

\bibitem{vdW81}
J.~van~der Wal.
\newblock {\em Stochastic Dynamic Programming}.
\newblock The Mathematical Centre, Amsterdam, 1981.

\bibitem{vanH77}
K.~M. van Hee, A.~Hordijk, and J.~van~der Wal.
\newblock Successive approximations for convergent dynamic programming.
\newblock In H.~Tijms and J.~Wessels, editors, {\em Markov Decision Theory},
  pages 183--211. The Mathematical Centre, Amsterdam, 1977.

\bibitem{wat89}
C.~J. C.~H. Watkins.
\newblock {\em Learning from Delayed Rewards}.
\newblock PhD thesis, Cambridge Univ., England, 1989.

\bibitem{Whit79}
P.~Whittle.
\newblock A simple condition for regularity in negative programming.
\newblock {\em J. Appl. Prob.}, 16:305--318, 1979.

\bibitem{YuB-mvipi}
H.~Yu and D.~P. Bertsekas.
\newblock A mixed value and policy iteration method for stochastic control with
  universally measurable policies.
\newblock LIDS Technical Report 2905, MIT, 2013.
\newblock Revised 2014; to appear in \emph{Math. Oper. Res.} Online at:
  \url{http://arxiv.org/abs/1308.3814}.

\bibitem{Zin}
M.~Zinsmeister.
\newblock Les derivations analytiques.
\newblock In {\em Seminaire de Probabilities XXIII}, Lecture Notes in Math.
  1372, pages 21--46. Springer-Verlag, Berlin and New York, 1989.

\end{thebibliography}
 
\thispagestyle{myheadings}  
\markboth{\rm Appendix \ref{appendix2}}{\rm Appendix \ref{appendix2}}
\begin{appendices}
\newcommand{\appsection}[1]{\let\oldthesection\thesection
  \renewcommand{\thesection}{
  \oldthesection}
  \section{#1}\let\thesection\oldthesection}

\appsection{Some Basic Properties of the (GC) Model in the Universal Measurability Framework} \label{appendix2}
\markboth{\rm Appendix \ref{appendix2}}{\rm Appendix \ref{appendix2}}

In this appendix, we prove some basic properties about policy costs for the (GC) total cost model, in the universal measurability framework. First, we prove \cref{lma-B}, which we recall here for convenience.

\lmaB*

\begin{proof}
(a) The proof arguments of \cite[Prop.\ 8.1]{bs} are applicable here. They show that for any given policy $\pi$, one can construct a Markov policy $\pi'$ such that for every $k \geq 0$, the marginal distribution of the state and control variables $(x_k, u_k)$ under $\pi'$ is the same as that under $\pi$ for the given initial state $x$. It then follows that $J_{\pi'}^+(x)=J_{\pi}^+(x)$ and $J_{\pi'}^-(x)=J_{\pi}^-(x)$. Since the difference between these two terms defines the total cost of $\pi'$ and $\pi$ at state $x$, we have $J_{\pi'}(x) = J_\pi(x)$.
\smallskip

\noindent (b) The statement is clearly true for a Markov policy $\pi$ by the Markov property and the optimality of $J^*$. 
Now suppose $\pi$ is a history-dependent policy. Let $\pi'$ be a Markov policy used in the proof of part (a), for the given state $x$ and policy $\pi$. Then, since the marginal distributions of $(x_k, u_k)$ are the same under $\pi$ and $\pi'$ for all $k \geq 0$, we have that 
$$ \E^\pi_x \left\{ \sum_{k=n}^\infty g(x_k, u_k) \right\} = \E^{\pi'}_x \left\{ \sum_{k=n}^\infty g(x_k, u_k) \right\} \qquad \text{and} \qquad 
    \E^\pi_x \left\{ J^*(x_n) \right\} =  \E^{\pi'}_x \left\{ J^*(x_n) \right\}. 
$$
Since $\E^{\pi'}_x \left\{ \sum_{k=n}^\infty g(x_k, u_k) \right\} \geq \E^{\pi'}_x \left\{ J^*(x_n) \right\}$, we obtain the first relation in part (b). 

To show the second relation in part (b), consider again the preceding Markov policy $\pi'$ and write it as $\pi'=(\mu_0', \mu_1', \ldots)$. 
Denote $\pi'_n = (\mu_n', \mu_{n+1}', \ldots)$ (which is also a Markov policy). Similarly to the preceding proof, because the marginal distributions of $(x_k,u_k)$ are the same under $\pi$ and $\pi'$ for all $k \geq 0$, we have
$$\E^\pi_x \left\{ \sum_{k=n}^\infty g_-(x_k, u_k) \right\} = \E^{\pi'}_x \left\{ \sum_{k=n}^\infty g_-(x_k, u_k) \right\} = \E^{\pi'}_x \left\{ J_{\pi'_n}^-(x_n) \right\}$$
and
$$   \oE^\pi_{n,x} \left\{ J^{*-}(x_n) \right\} =  \oE^{\pi'}_{n,x} \left\{ J^{*-}(x_n) \right\}.$$ 
By the optimality of $J^{*-}$, $J^{*-} \leq J_{\pi'_n}^- \in \M(S)$, and therefore, by the definition (\ref{eq-outerint1}) of outer integral, 
$$ \E^{\pi'}_x \left\{ J_{\pi'_n}(x_n) \right\}  \geq \oE^{\pi'}_{n,x} \left\{ J^{*-}(x_n) \right\}.$$
Combining the preceding three relations, we obtain $\E^\pi_x \left\{ \sum_{k=n}^\infty g_-(x_k, u_k) \right\} \geq \oE^\pi_{n,x} \left\{ J^{*-}(x_n) \right\}$.
\end{proof}

Let us also verify another basic fact about policy costs in the (GC) model:

\begin{lem} \label{lma-DPpolicy}
For any stationary policy $\mu$, $J_\mu=T_{\mu}(J_\mu)$.
\end{lem}
\begin{proof}
We can write $J_\mu = J_\mu^+ - J_\mu^-$ and $T_{\mu}(J_\mu) = T_{\mu}^+(J_\mu^+) - T_{\mu}^-(J_\mu^-)$, where $J_\mu^+$, $T_{\mu}^+$ (resp.\ $J_\mu^-,  T_{\mu}^-$) are the total cost of $\mu$ and the dynamic programming operator associated with $\mu$, respectively, in the total cost problem with nonnegative one-stage cost function $g_+$ (resp.\ $g_-$). By \cite[Prop.\ 9.9]{bs}, 
$J_\mu^+ = T_{\mu}^+(J_\mu^+)$ and $J_\mu^- = T_{\mu}^-(J_\mu^-)$, the conclusion then follows.
\end{proof}

\appsection{Optimality Properties of the (GC) Model in the Universal Measurability Framework} \label{appendix1}
\markboth{\rm Appendix \ref{appendix1}}{\rm Appendix \ref{appendix1}}

In this appendix, we establish that in the universal measurability framework, the (GC) total cost model has the optimality properties stated by Theorems~\ref{thm-m1},~\ref{thm-m2}. We recall the two theorems here for convenience:

\thmoptone*

\thmopttwo*

Most of these optimality results can be derived by following the lines of analyses given in the paper \cite{ShrB79} and the book \cite[Chap.\ 9.2-9.3]{bs}, although these works addressed infinite-horizon undiscounted problems with purely positive or negative one-stage costs only. We will not repeat the entire analyses here. In what follows, we first explain in Section~\ref{secB.1} the key arguments needed in order to set up the stage for proofs. We then outline the proofs for Theorem~\ref{thm-m1}(a), Theorem~\ref{thm-m2}(a) and Theorem~\ref{thm-m2}(b) (see Sections~\ref{secB.2.1},~\ref{secB.2.2}, respectively); and we give detailed proofs for Theorem~\ref{thm-m1}(b), the part of Theorem~\ref{thm-m2}(a) relating to the construction of a nonrandomized near-optimal policy, and Theorem~\ref{thm-m2}(c) (see Sections~\ref{secB.2.3},~\ref{appsec-prf-thm2a},~\ref{secB.2.5}, respectively). 

Let us start the discussion with the lower semi-analyticity property of $J^*$ stated in Theorem~\ref{thm-m1}(a). If $T^n(\0) \to J^*$, we can conclude immediately by \cite[Lemma 7.30(2)]{bs} that $J^*$ must be lower semi-analytic. Unfortunately, as in the positive model (P), value iteration need not converge for the (GC) model, so a direct proof that does not rely on the convergence of value iteration is needed. It is achieved by setting up a corresponding deterministic total cost problem~\cite[Chap.\ 9.2]{bs}, to be referred to as (DM), on the space of marginal distributions of the state and control variables. The problem (DM), roughly speaking, is constructed to have the following properties:
\begin{itemize}
\item[(i)] it is equivalent to the original problem in terms of policy costs; and 
\item[(ii)] it allows direct applications of measurable selection results, in particular a Jankov-von Neumann type of selection theorem, which states that partial minimization of a lower semi-analytic function 
results in a lower semi-analytic function, together with a universally measurable minimizer or $\epsilon$-minimizer~\cite[Props.\ 7.47, 7.50]{bs}.
\end{itemize}
To prove the optimality properties of the original problem, we first analyze (DM) as an intermediate step, and we then transcribe the results from (DM) to the original problem, using the correspondences between the two problems.

\subsection{A Deterministic Model and its Relation to the Original Problem} \label{secB.1}

Recall that $\P(S)$ (resp.\ $\P(S \times C)$) denotes the set of probability measures on the Borel $\sigma$-algebra $\B(S)$ (resp.\ $\B(S \times C)$) of the state space $S$ (resp.\ the state-control space $S \times C$). The graph of the control constraint $U$ of a given total cost (GC) problem is the analytic set $\Gamma = \{ (x,u) \mid x \in S, u \in U(x) \}$.

\begin{definition}[{\cite[Defs.\ 9.4-9.6]{bs}}]  \label{def-DM} 
For a total cost (GC) problem given in Section~\ref{sec2.1}, 
the \emph{corresponding deterministic control model (DM)} is defined as follows. \\*[0.1cm]
\noindent (i) The state space, control space, and model parameters are given by:
\begin{itemize}
\item[$\bullet$] State space $\P(S)$ and control space $\P(S \times C)$.
\item[$\bullet$] Control constraint $\bar U$, which maps each state $p \in \P(S)$ to a set $\bar U(p)$ of feasible controls at $p$, defined as
$$ \bar U(p) : = \big\{ \q \in \P(S \times C) \ \big|  \ \q(B \times C) = p(B), \ \q(\Gamma) = 1, \, \forall \, B \in \B(S) \big\}.$$
I.e., the controls at state $p$ are those probability measures on $S \times C$ that have $p$ as the marginal on $S$ and assign probability one to the graph $\Gamma$ of the original control constraint $U$. (Like $U$, the graph of $\bar U$ is an analytic set in $\P(S) \times \P(S \times C)$; see the proof of~\cite[Lemma 9.1]{bs}.)
\item[$\bullet$] System function $\bar f: \P(S \times C) \to \P(S)$, which maps each control $\q \in \P(S \times C)$ to a state $\bar f(\q) \in \P(S)$, defined according to the state transition kernel $q$ of the original problem as
$$ \bar f(\q)(B) : = \int_{S \times C} q(B \!\mid x, u) \, \q\big(d(x,u)\big), \qquad B \in \B(S).$$
This function $\bar f$ specifies how the states evolve in (DM). 
\item[$\bullet$] One-stage cost function $\bar g : \P(S \times C) \to [-\infty, + \infty]$, defined by the one-stage cost function $g$ of the original problem as
$$ \bar g(\q)  : = \bar g_+(\q) - \bar g_-(\q)$$ 
where
$$      \bar g_+(\q) : =  \int_{S \times C} g_+(x,u) \, \q \big(d(x,u)\big), \qquad \bar g_-(\q):=  \int_{S \times C} g_-(x,u) \, \q \big(d(x,u)\big),$$
with the convention $\infty - \infty = \infty$. (Like $g$, the functions $\bar g_+$, $-\bar g_-$ and $\bar g$ are all lower semi-analytic~\cite[Prop.\ 7.48 and Lemma 7.30]{bs}.)
\end{itemize}
\noindent (ii) A \emph{policy for (DM)} is a sequence of mappings $\bar \pi = (\bar \mu_0, \bar \mu_1, \ldots)$ such that for each $k \geq 0$, $\bar \mu_k : \P(S) \to \P(S \times C)$ and $\bar \mu_k(p) \in \bar U(p)$ for every $p \in \P(S)$ (i.e., $\bar \mu_k$ maps each state $p$ to a feasible control at $p$). The set of all policies for (DM) is denoted by $\bar \Pi$. \\
\noindent (iii) Given an initial state $p_0 \in \P(S)$, applying a policy $\bar \pi$ in (DM) generates recursively a sequence of state and control paris $(p_0, \q_0), (p_1, \q_1), \ldots$ as
\begin{equation}
  \q_k = \bar \mu_k(p_k), \qquad p_{k+1}=\bar f(\q_k), \quad k = 0, 1, \ldots,
\end{equation}  
where $p_k$ is the state at time $k$, generated according to the system function $\bar f$.
\end{definition} 
\smallskip

We now discuss the relations between the original problem and the deterministic model (DM) just defined (cf.\ \cite[Def.\ 9.9 and Prop.\ 9.2]{bs}):
\begin{itemize}
\item[(a)] Given $p_0 \in \P(S)$ and a policy $\pi$ of the original problem, let $p_k \in \P(S)$, $\q_k \in \P(S \times C)$, $k \geq 0$, be the marginal distributions of the state $x_k$, the state-control pairs $(x_k,u_k)$, respectively, at time $k$, under the policy $\pi$, with the initial state distribution being $p_0$. We can define a corresponding policy $\bar \pi = (\bar \mu_0, \bar \mu_1, \ldots)$ for (DM) such that these marginal distributions $\q_k, p_{k+1}$ are exactly the controls and states that would be generated in (DM) according to Def.~\ref{def-DM}(iii), if $\bar \pi$ is applied and the initial state is $p_0$.
\item[(b)] Conversely, given $p_0 \in \P(S)$ and a policy $\bar \pi$ of (DM), let $\q_k$, $p_{k+1}$, $k \geq 0$ be the controls and states generated according to Def.~\ref{def-DM}(iii) under $\bar \pi$, for the initial state $p_0$. Then 
we can use the decomposition result \cite[Prop.\ 7.27 and Cor.\ 7.27.2]{bs} to define from the sequence $\{ (p_k,\q_k)\}$ a corresponding Markov policy $\pi$ for the original problem, such that these $p_k, \q_k$ would be exactly the marginal distributions of $x_k$, $(x_k,u_k)$, respectively, in the original problem if $\pi$ is applied and the initial state distribution is $p_0$.
\end{itemize}
Given these correspondences between the two problems, under the (GC) condition~(\ref{eq-tmodel}) on the original problem, we can define the optimal total cost function $\bar J^*$ for (DM) as follows.

\begin{definition}[{Cf.\ \cite[Def.\ 9.6]{bs}}] 
In (DM), for a state $p_0 \in \P(S)$, the \emph{cost of a policy} $\bar \pi$ at $p_0$ is defined as
$$ \bar J_{\bar \pi}(p_0) = \bar J_{\bar \pi}^+(p_0)  - \bar J_{\bar \pi}^-(p_0)$$
(with the convention $\infty-\infty=\infty$), where
$$ \bar J_{\bar \pi}^+(p_0)  :=\sum_{k=0}^\infty \bar g_+(\q_k), \qquad \bar J_{\bar \pi}^-(p_0)  :=\sum_{k=0}^\infty \bar g_-(\q_k), $$
and the sequence $\{\q_k\}$ is generated by $\bar \pi$ according to Def.~\ref{def-DM}(iii). 
The \emph{optimal cost} at $p_0$ is
$$ \bar J^*(p_0) = \inf_{\bar \pi \in \bar \Pi} \bar J_{\bar \pi} (p_0).$$
\end{definition}
\smallskip

For $x \in S$, denote by $\delta_x$ the Dirac measure that assigns probability $1$ to the point $x$.
From the correspondence between the two models described in (b) earlier, we see that if the original problem satisfies the (GC) condition~(\ref{eq-tmodel}), then in its corresponding (DM), for all $x \in S$ and policies $\bar \pi$ of (DM), $\bar J_{\bar \pi}^-(\delta_x)$ is finite and $\bar{J}_{\bar \pi}(\delta_x) > - \infty$, and moreover, $\bar J^*(\delta_x) > - \infty$. 

Furthermore, by the relations described in (a)-(b) earlier, the total costs of the policies in the two problems correspond as well:
\begin{itemize}
\item[($\text{a}'$)] Given $x \in S$ and a policy $\pi$ of the original problem, there exists a policy $\bar \pi$ of (DM) such that $J_\pi(x) = \bar J_{\bar \pi}(\delta_x)$. 
\item[($\text{b}'$)] Conversely, given $x \in S$ and a policy $\bar \pi$ of (DM), there exists a Markov policy $\pi$ of the original problem such that $J_\pi(x) = \bar J_{\bar \pi}(\delta_x).$ 
\end{itemize}
Thus we must have $J^*(x) = \bar J^*(\delta_x)$ for all $x \in S$.

Now a crucial property of (DM) is that its optimal cost function $\bar J^*$ is lower semi-analytic. This can be shown by writing $\bar J^*$ as the result of partial minimization of a lower semi-analytic function as follows. Let $\Delta$ consist of all points $(p_0, \q_0, \q_1, \ldots)$, called \emph{admissible sequences} \cite[Def.\ 9.7]{bs}, such that 
$$p_0 \in \P(S), \qquad (\q_0, \q_1, \ldots) \in \Delta_{p_0}, $$
where $\Delta_{p_0}$ is the set of all control sequences $(\q_0, \q_1, \ldots)$ that can be generated by some policy of (DM) for the initial state $p_0$. Define a function $G: \Delta \to [-\infty, + \infty]$ by 
$$ G \big(p_0, \q_0, \q_1, \ldots \big) : = \sum_{k=0}^\infty \bar g_+(\q_k) - \sum_{k=0}^\infty \bar g_-(\q_k) $$
(with $\infty-\infty=\infty$). 
Then $\Delta$ is an analytic subset of $\P(S) \times \big(\P(S \times C) \big)^\infty$ \cite[Lemma 9.1]{bs}; the function $G$ is lower semi-analytic \cite[Lemma 7.30(4)]{bs}; and by the definition of the optimal cost function $\bar J^*$, it can be written equivalently as the result of partial minimization of $G$:
\begin{equation} \label{eq-bJ}
    \bar J^*(p_0) = \inf_{(\q_0, \q_1, \ldots) \in \Delta_{p_0}} G \big(p_0, \q_0, \q_1, \ldots \big).
\end{equation}
Hence, by \cite[Prop.\ 7.47]{bs}, $\bar J^*$ is lower semi-analytic. 

Moreover, by a measurable selection theorem \cite[Prop.\ 7.50]{bs}, for any $\epsilon > 0$, we can select a measurable $\epsilon$-minimizer for the optimization problem in (\ref{eq-bJ}). More precisely, there exists a universally measurable mapping $\psi: \P(S) \to \big(\P(S \times C)\big)^\infty$ such that
for all $p_0 \in \P(S)$,
$$  \psi(p_0) \in \Delta_{p_0} \ \ \ \text{and} \ \ \ G\big(p_0, \psi(p_0) \big) \leq \begin{cases}  
    \bar J^*(p_0)+\epsilon, & \text{if} \  \bar J^*(p_0) > - \infty; \\
    - 1/\epsilon, & \text{if} \ \bar J^*(p_0) = - \infty.  \end{cases} $$
As noted earlier, $\bar J^*(\delta_x) > - \infty$ for all $x \in S$ since the original problem is a (GC) problem. Therefore,
\begin{equation} \label{eq-bJ-min}
  \psi(\delta_x) \in \Delta_{\delta_x} \ \ \ \text{and} \ \ \ G\big(\delta_x, \psi(\delta_x) \big) \leq \bar J^*(\delta_x)+\epsilon, \qquad \forall\, x \in S.
\end{equation}  
Furthermore, by \cite[Prop.\ 7.50(b)]{bs}, $\psi$ can be chosen to attain the infimum in (\ref{eq-bJ}) whenever this is possible; i.e.,
\begin{equation} \label{eq-bJ-min2}
   G\big(p_0, \psi(p_0) \big) = \bar J^*(p_0), \qquad \forall \, p_0 \in \I,
\end{equation}
where 
\begin{equation} \label{eq-bJ-min3}
 \I = \Big\{ \, p \in \P(S) \, \, \Big| \,  \, \exists \, (\q_0, \q_1, \ldots) \in \Delta_{p} \text{ with } \bar J^*(p) = G \big(p, \q_0, \q_1, \ldots \big)  \,  \Big\}.
\end{equation}

\subsection{Proof Arguments for Theorems~\ref{thm-m1}, \ref{thm-m2}} \label{secB.2}

\subsubsection{$J^*$ is Lower Semi-analytic} \label{secB.2.1}

The correspondence between the original problem and (DM), as described in (a)-(b) and ($\text{a}'$)-($\text{b}'$) earlier, implies that $J^*(x) = \bar J^*(\delta_x)$ for all $x \in S$. In other words, $J^*$ is the composition of the function $\bar J^*$, which was just proved to be lower semi-analytic, with the function $x \mapsto \delta_x$, which is a homeomorphism \cite[Cor.\ 7.21.1]{bs}. Then, since the composition $f_1 \circ f_2$ is lower semi-analytic for a lower semi-analytic function $f_1$ and a Borel-measurable function $f_2$ \cite[Lemma 7.30(3)]{bs}, we can conclude that $J^*$ is a lower semi-analytic function, as stated in Theorem~\ref{thm-m1}(a).

\subsubsection{Existence of an $\epsilon$-Optimal Semi-Markov Policy} \label{secB.2.2}

We consider now the statements in Theorem~\ref{thm-m2} about the existence of $\epsilon$-optimal or optimal policies. The measurable $\epsilon$-minimizer $\psi$ obtained in Eq.~(\ref{eq-bJ-min}) for (DM) gives us the universally measurable mapping $\phi(x)=\psi(\delta_x)$ \cite[Prop.\ 7.44]{bs}, which maps $x \in S$ to a sequence of probability measures in $\big(\P(S \times C)\big)^\infty$. Using the decomposition result \cite[Prop.\ 7.27]{bs} and using also the relations (a)-(b) and ($\text{a}'$)-($\text{b}'$) between the two models mentioned earlier, one can construct from $\phi$ an $\epsilon$-optimal, possibly randomized, semi-Markov policy for the original problem. The method of this construction is the same as the one used in the second half of the proof of \cite[Prop.~9.20, p.~240]{bs}. That same part of the proof, together with Eqs.~(\ref{eq-bJ-min})-(\ref{eq-bJ-min3}), also establishes Theorem~\ref{thm-m2}(b) by constructing an optimal, possibly randomized, semi-Markov policy for the original problem. 

We still have not shown that an $\epsilon$-optimal policy can be chosen to be nonrandomized, as stated in Theorem~\ref{thm-m2}(a). We shall prove this later in Appendix~\ref{appsec-prf-thm2a}.

\subsubsection{The Optimality Equation $J^*= T(J^*)$} \label{secB.2.3}
   
To show that the optimality equation $J^* = T(J^*)$ holds, we can establish first the optimality equation for the corresponding deterministic model (DM) and then transcribe the result from (DM) to the original problem. Alternatively, we can prove $J^* = T(J^*)$ directly, using the fact that there exists an $\epsilon$-optimal policy for the original problem, which we just established. We give such a proof below. The line of analysis is similar to the one for \cite[Theorem 6.3]{Fein02}. The existence of an $\epsilon$-optimal policy is used to prove $J^* \leq T(J^*)$.

\begin{proof}[Proof of Theorem~\ref{thm-m1}(b)]
We prove $J^* \leq T(J^*)$ first. Let $\epsilon >0$. As we showed earlier, there exists a policy $\hat \pi$ which is $\epsilon$-optimal for all states:
$$ J_{\hat \pi}(x) \leq J^*(x) + \epsilon, \qquad \forall \, x \in S.$$
By the selection theorem~\cite[Prop.\ 7.50]{bs}, there exists a nonrandomized stationary policy $\mu$ such that
$$ T_\mu(J^*) \leq T(J^*) + \epsilon.$$
Then, for the policy $\pi =(\mu, \hat \pi)$ (the concatenation of $\mu$ with $\hat \pi$), we have
$$ J^* \leq J_\pi = T_\mu(J_{\hat \pi}) \leq T_\mu(J^*) + \epsilon \leq T(J^*) + 2 \epsilon.$$
Since $\epsilon$ is arbitrary, we obtain $J^* \leq T(J^*)$.

We now prove $J^* \geq T(J^*)$. Consider an arbitrary state $x$. Take a policy $\hat \pi=(\hat \pi_0, \hat \pi_1, \ldots)$ with
$$ J_{\hat \pi}(x) \leq J^*(x) + \epsilon.$$
We also have $J_{\hat \pi}(x) \geq T(J^*)(x)$, since
\begin{align*}
    J_{\hat \pi}(x) & = \, \E^{\hat \pi}_x \left\{ g(x_0, u_0)\right\} + \E^{\hat \pi}_x \left\{ \sum_{k=1}^\infty g(x_k, u_k)\right\}  \\
      & \geq \, \E^{\hat \pi}_x \left\{ g(x_0, u_0)\right\} + \E^{\hat \pi}_x \left\{ J^*(x_1) \right\} \\
      & = \, \int_{C} \left\{ g(x, u) + \int_S J^*(x') \, q(dx' \!\mid x, u) \right\} \, \hat\pi_0(du \!\mid x) \\
      & \geq \, \inf_{u \in U(x)} \, \left\{ g(x, u) + \int_S J^*(x') \, q(dx' \!\mid x, u) \right\}  = \, T(J^*)(x),
\end{align*}
where the first inequality follows from Lemma~\ref{lma-B}(b).
Combining the two relations, we obtain $J^*(x) \geq T(J^*)(x) - \epsilon$. Since $x$ and $\epsilon$ are arbitrary, we obtain $J^* \geq T(J^*)$. So by the first part of the proof, $J^* = T(J^*)$.
\end{proof}

\subsubsection{Existence of an $\epsilon$-Optimal Nonrandomized Semi-Markov Policy} \label{appsec-prf-thm2a}

We showed earlier that an $\epsilon$-optimal, possibly randomized, semi-Markov policy exists for the (GC) model. We will now prove that this policy can be chosen to be nonrandomized, thus establishing Theorem~\ref{thm-m2}(a) completely. The construction of this policy is similar to the one given in the proof of \cite[Prop.~9.20, p.~239]{bs} for the negative costs model (N), except for a few details due to the difference between the negative costs model (N) and the (GC) model. To account for this difference, we will need in the proof the convergence property of value iteration in the (GC) model; in particular, we will use the fact that value iteration converges to $J^*$ from above if it starts from $J^{*+}$, for instance. 
We give a complete proof below.

\begin{proof}[Proof of Theorem~\ref{thm-m2}(a)]
Let $\epsilon > 0$ be given. Recall that $J^{*+}$ is the optimal cost function of the positive costs problem with one-stage cost function $g_+$. 
For $k \geq 0$, let $J_k = T^k(J^{*+})$ and let
$$ A_k = \big\{ x \in S \mid J_k(x) \leq J^*(x) + \epsilon/4 \big\}.$$
Since $J_k, J^*$ are lower semi-analytic, the sets $A_k$ are universally measurable, and since $J_k \downarrow J^*$ (Theorem~\ref{thm-vi}), we have $A_0 \subset A_1 \subset \cdots$ and $S = \bigcup_{k \geq 0} A_k$.
For each $k \geq 0$, consider the $(k+1)$-stage problem corresponding to $T^k(J^{*+})$, which has $J^{*+}$ as the terminal cost function at the last stage. Applying the same argument as in the proof of \cite[Prop.\ 8.3($\text{F}^+$)]{bs} for finite-horizon problems, we have that there exists a nonrandomized Markov policy $\pi$ whose $k$-stage cost function $J_{\pi,k}$ satisfies that
$$   J_{\pi,k}(x) + \E^\pi_x \{ J^{*+}(x_{k}) \} \leq J_k(x) + \epsilon/4, \qquad \forall \, x \in S.$$
For states in the set $A_k$, this implies the bound,
\begin{equation} \label{eq-app-prf0}
  J_{\pi,k}(x) + \E^\pi_x \{ J^{*+}(x_{k}) \} \leq J^*(x) + \epsilon/2, \qquad \forall \, x \in A_k.
\end{equation}  
Let us write this policy $\pi$, which is associated with $k$, as $\pi^{(k)}=(\pi^{(k)}_0, \pi^{(k)}_1, \ldots)$.

Now let $\pi^+=(\pi^+_0, \pi^+_1, \ldots)$ be a nonrandomized $\epsilon/2$-optimal Markov policy for the positive costs problem with one-stage cost function $g_+$. (The existence of such a policy follows from \cite[Prop.\ 9.19(P)]{bs}.) Since $g \leq g_+$, the cost of $\pi^+$ in the original problem can be bounded by
$$ J_{\pi^+}(x) \leq J_{\pi^+}^+(x) \leq J^{*+}(x) + \epsilon/2, \qquad \forall \, x \in S.$$
Consequently, on each set $A_k$, $k \geq 0$, we have, by Eq.~(\ref{eq-app-prf0}), that for the associated policy $\pi^{(k)}$,
\begin{equation} \label{eq-app-prf1} 
  J_{\pi^{(k)},k}(x) +  \E^{\pi^{(k)}}_x \big\{ J_{\pi^+}(x_k) \big\} \leq J^*(x) + \epsilon, \qquad \forall \, x \in A_k.
\end{equation}  

Define a semi-Markov policy 
$$\pi^\epsilon = \big(\mu_0(du_0 \!\mid x_0), \, \mu_1(du_1 \!\mid x_0, x_1), \, \ldots, \,\mu_n(d u_n \!\mid x_0, x_n), \, \ldots \big) $$ 
as follows. 
Partition the state space $S$ into a countable number of disjoint sets, $A_k \setminus A_{k-1}$, $k \geq 0$, with $A_{-1} = \emptyset$.
For initial states in the set $A_k \setminus A_{k-1}$, let $\pi^\epsilon$ be identical to the concatenation of the first $k$-stage of the policy $\pi^{(k)}$ and the policy $\pi^+$; more precisely, for $x_0 \in A_k \setminus A_{k-1}$, let
$$
   \mu_0(du_0 \!\mid x_0) = \pi^{(k)}_0(du_0 \!\mid x_0), \ \ \  \ldots, \ \ \ \mu_{k-1}(du_{k-1} \!\mid x_0, x_{k-1}) = \pi^{(k)}_{k-1}(du_{k-1} \!\mid x_{k-1}),$$
and let
$$   \mu_n( du_{n} \!\mid x_0, x_{n}) = \pi^+_n(d u_n \! \mid x_n), \ \ \ n \geq k. $$   
Then $\pi^\epsilon$ is a universally measurable semi-Markov policy, and it is nonrandomized since $\pi^+$ and $\pi^{(k)}, k \geq 0$ are all nonrandomized. Furthermore, for $x \in A_k \setminus A_{k-1}$, by the Markov property,
$$J_{\pi^\epsilon}(x) = J_{\pi^{(k)},k}(x) +  \E^{\pi^{(k)}}_x \big\{ J_{\pi^+}(x_k) \big\} \leq  J^*(x) + \epsilon,$$
where the inequality follows from Eq.~(\ref{eq-app-prf1}). Therefore, $J_{\pi^\epsilon}(x)  \leq J^*(x) + \epsilon$ for all $x \in S$; i.e., $\pi^\epsilon$ is $\epsilon$-optimal.
\end{proof}

\subsubsection{Existence of $\epsilon$-Optimal Markov or Optimal Stationary Policies when $J^* \geq 0$} \label{secB.2.5}

Finally, we verify the statements in Theorem~\ref{thm-m2}(c) about the existence of $\epsilon$-optimal Markov policies or optimal stationary policies when $J^* \geq 0$. The conclusions of Theorem~\ref{thm-m2}(c) are indeed the same as those known to hold for the positive costs model, and the proofs for the two models are also almost the same (cf.\ the proof of~\cite[Prop.\ 9.19(P)]{bs}). For clarity, however, we include the proof below. It uses the optimality equation $J^*=T(J^*)$ and the assumption $J^* \geq 0$.

\begin{proof}[Proof of Theorem~\ref{thm-m2}(c)]
Since $J^* \geq 0$ by assumption, $T(J^*)(x) = J^*(x) \geq 0 > - \infty$ for each $x \in S$. So by the selection theorem~\cite[Prop.\ 7.50]{bs}, given any $\epsilon > 0$, there exist stationary nonrandomized policies $\mu_k, k \geq 0$, such that
\begin{equation} \label{eq-prf-m2c-a}
   T_{\mu_k}(J^*) \leq T (J^*) + 2^{-k-1} \epsilon = J^* + 2^{-k-1} \epsilon.
\end{equation}   
Consider the nonrandomized Markov policy $\pi = (\mu_0, \mu_1, \ldots)$. 
By the monotonicity of the mappings $T_{\mu_0}, T_{\mu_1}, \ldots$, we have that for $n \geq 1$,
\begin{align*}
   \big(T_{\mu_0} \circ T_{\mu_1} \circ \cdots  T_{\mu_n} \big) (\0)  & \leq \big(T_{\mu_0} \circ T_{\mu_1} \circ \cdots  T_{\mu_n} \big) (J^*)  \\
      & \leq \big(T_{\mu_0} \circ T_{\mu_1} \circ \cdots  T_{\mu_{n-1}} \big) (J^*)  + 2^{-n-1} \epsilon   \\
      & \leq \big(T_{\mu_0} \circ T_{\mu_1} \circ \cdots  T_{\mu_{n-2}} \big) (J^*)  + (2^{-n} + 2^{-n-1}) \epsilon \\
       & \qquad \quad \vdots   \\
       & \leq J^* + \sum_{k=0}^{n}  2^{-k-1} \epsilon   \leq J^*  + \epsilon,
\end{align*}       
where the assumption $J^* \geq 0$ is used in the first inequality, and Eq.~(\ref{eq-prf-m2c-a}) is used in the subsequent inequalities. 
Since $J_{\pi} = \lim_{n \to \infty} \big(T_{\mu_0} \circ T_{\mu_1} \circ \cdots  T_{\mu_n} \big) (\0)$ in the (GC) model, it follows that $\pi$ is an $\epsilon$-optimal policy. Hence there exists an $\epsilon$-optimal nonrandomized Markov policy if $J^* \geq 0$.

We now prove the second statement in Theorem~\ref{thm-m2}(c). If there exists an optimal policy for each state, by Lemma~\ref{lma-B}(a), there exists for each given state $x \in S$, a Markov policy $\pi=(\mu_0, \mu_1, \ldots)$ that is optimal for $x$. Denoting $\pi' = (\mu_1, \mu_2, \ldots)$, we have by the Markov property and the monotonicity of $T_{\mu_0}$,
$$ T_{\mu_0}(J^*)(x) \leq T_{\mu_0}(J_{\pi'})(x) = J_{\pi}(x) = J^*(x) = T(J^*)(x).$$
This shows that for each $x$, the infimum defining $T(J^*)(x)$ is attained. Hence, by the selection theorem~\cite[Prop.\ 7.50]{bs}, there exists a (universally measurable) nonrandomized stationary policy $\mu$ satisfying $T_{\mu}(J^*) = J^*$. Using the assumption $J^* \geq 0$ and the monotonicity of $T_\mu$, we obtain
\begin{equation} \label{eq-prf-m2c-b}
  J_\mu = \lim_{n \to \infty} T_\mu^n(\0) \leq \lim_{n \to \infty} T_\mu^n(J^*) = J^*,
\end{equation}  
so we must have $J_\mu = J^*$, i.e., $\mu$ is optimal. Since $\mu$ is nonrandomized and stationary, this establishes the desired result.

We now prove the last statement in Theorem~\ref{thm-m2}(c). Equation~(\ref{eq-prf-m2c-b}) already shows that when $J^* \geq 0$, any stationary policy $\mu$ satisfying $T_{\mu}(J^*) = J^*$ is optimal. Conversely, if $\mu$ is an optimal stationary policy, then $J_\mu = J^*$, so by Lemma~\ref{lma-DPpolicy}, $T_{\mu}(J^*) =  J^*$. This completes the proof.
\end{proof}

\appsection{About Transfinite Value Iteration and an Important Lemma for Proving Theorem~\ref{thm-relate2P}} \label{appendix3}
\markboth{\rm Appendix \ref{appendix3}}{\rm Appendix \ref{appendix3}}

In this appendix we first give an informal introduction of ordinals, in connection with transfinite value iteration. 
We then focus on an important lemma needed to establish Theorem~\ref{thm-relate2P} on the convergence of transfinite value iteration in the case where the optimal cost function is nonnegative.

\subsection{Ordinals and Transfinite Value Iteration} \label{appendix3-1}
 
Ordinals are well-ordered sets chosen in such a way that each well-ordered set is order-isomorphic to one and only one ordinal; in other words, each ordinal represents an order type. We can linearly order these order types so that a more complicated ordering is represented by a larger ordinal. The book \cite[Chaps.\ A.3]{Dud02} gives a good introduction of the theory of ordinals, and the informal explanations we give here are partly based on this reference.
 
Intuitively, we can construct ordinals starting from the smallest set, the empty set, as follows. Ordinals with finite elements, named after the number of their elements, are:
$$ \text{ordinal $0$}: \emptyset, \quad \text{ordinal $1$}:  \{ \emptyset\}, \quad \text{ordinal $2$}: \big\{ \emptyset, \{ \emptyset\} \big\}, \quad \text{ordinal $3$}: \big\{  \emptyset, \{ \emptyset\}, \{ \emptyset, \{ \emptyset\} \}  \big\}, \quad \ldots.$$
In general, each element of an ordinal is a set itself, the \emph{successor ordinal} $\xi+1$ of an ordinal $\xi$ is given by the set $\xi \cup \{\xi\}$, and any ordinal is well-ordered by the relation $\in$. 
The preceding finite ordinals are order-isomorphic to the simple finite sets $\emptyset, \{1\}$, $\{1,2\}, \ldots$, respectively, with the usual ordering. 
The union of any set of ordinals is also an ordinal. 
The first infinite ordinal, denoted by $\omega$, is the union of the finite ordinals $1,2,\ldots$, 
and order-isomorphic to the set of positive integers. 
The ordinal $\omega$ is equal to the union of its elements---such an ordinal is called a \emph{limit ordinal}.
It can be shown that an ordinal is either a limit ordinal or a successor of some ordinal $\xi$, which is its largest element. In other words, with the ``$+1$'' and union operations, increasingly larger ordinals are constructed successively. (This process goes on indefinitely; there is no largest ordinal.)

The transfinite value iteration discussed in Section~\ref{sec4.1} is a recursion defined on the well-ordered set $\{0, 1, \ldots, \omega_1\}$, with $\omega_1$ being the smallest ordinal that has an uncountably infinite number of elements. Iterations associated with those limit ordinals involve taking various limits of the ``past'' value iteration iterates, roughly speaking, whereas iterations associated with those ordinals of the form $\xi+1$ are similar to ordinary value iteration with the dynamic programming operator. For example, consider the transfinite value iteration (\ref{eq-transvi0})-(\ref{eq-transvi1}) given in Section~\ref{sec4.1}, which involves the mapping $\T(J) = \max\{ J, \, T(J)\}$.  Let $J_0$ be the initial function, and denote $J_{\xi} = \T^{\xi}(J_0)$ for an ordinal $\xi < \omega_1$. For $\xi < \omega$ (the first infinite ordinal), the iterates $J_\xi$ coincide with the iterates produced by ordinary fixed point iteration with $\T$: 
$$ J_1 = \T(J_0), \quad \ldots, \quad J_{n+1} = \T(J_n), \ \ldots,$$
which form an increasing sequence of functions by the increasing property of $\T$, $\T(J) \geq J$.
For the limit ordinal $\omega$, $J_\omega$ can be expressed as
$$ J_\omega = \T \Big( \lim_{n \to \infty} J_n \Big) = \T \Big( \sup_{n \geq 0} J_n \Big),$$
and for the ordinal $\omega+1$, by the increasing property of $\T$, $J_{\omega+1}$ is simply given by
$$ J_{\omega+1} = \T \big(J_\omega \big).$$
Ordinarily, we stop value iteration at the ordinal $\omega$. But when the limit function $\lim_{n \to \infty} J_n$ does not equal $J^*$ (i.e., ordinary value iteration fails to converge to $J^*$), then transfinite value iteration can be used to analyze the behavior of value iteration, and it can shed light on the nature of value iteration, as it does for the positive costs model~\cite{MS92}.

\subsection{A Lemma for Convergence of Transfinite Value Iteration} \label{appendix3-2}
We now state and outline the proof of an important lemma needed to establish Theorem~\ref{thm-relate2P} on the convergence of transfinite value iteration. Recall that Theorem~\ref{thm-relate2P} states that $\T^{\omega_1}(\0) = J^*$ if $J^* \geq 0$, where 
$\T(J) = \max\{ J, \, T(J)\}$. Let $J_{\omega_1} = \T^{\omega_1}(\0)$. 
In proving \cref{thm-relate2P}, we used the following lemma:

\begin{lem} \label{lma-rel2P}
The function $J_{\omega_1}$ is lower semi-analytic and satisfies $$J_{\omega_1} = \T(J_{\omega_1}) \geq T(J_{\omega_1}).$$
\end{lem}

\old{From \cref{lma-rel2P}, we obtain \cref{thm-relate2P} immediately, by combining this lemma with Prop.~\ref{prp-relate2P} and the assumption $J^* \geq 0$, as follows. 
Since $J^* \geq 0$, using the definition of transfinite value iteration and the monotonicity of $T$, and using also the fact $J^* = T(J^*)$, we have
$$J_{\omega_1}  = \T^{\omega_1}(\0) \leq \T^{\omega_1}(J^*) = J^*.$$
On the other hand, by the nondecreasing property of the transfinite value iterates, we have $J_{\omega_1} = \T^{\omega_1}(\0) \geq \T^0(\0) = \0$, 
and by \cref{lma-rel2P}, we also have $J_{\omega_1} \geq T(J_{\omega_1})$ and $J_{\omega_1}$ is lower semi-analytic. So applying Prop.~\ref{prp-relate2P} with $J = J_{\omega_1}$, we can conclude that $J_{\omega_1} \geq J^*$. Therefore, we must have $J_{\omega_1} = J^*$, as stated by Theorem~\ref{thm-relate2P}.}

Lemma~\ref{lma-rel2P} can be established by using the proof arguments given in Maitra and Sudderth's analysis of transfinite value iteration for the positive costs model~\cite{MS92}. In the rest of this appendix, we will explain the main proof steps---we do not include a detailed proof because most arguments are the same as in~\cite{MS92}. (The paper \cite{MS92} considers the reward maximization framework and the negative rewards model, which is equivalent to the positive costs model in the cost minimization framework we use. However, because of this difference, some of our notation and definitions below will appear to be different from those in~\cite{MS92}.) 

First, note that if for some ordinal $\xi < \omega_1$, $\T^{\xi}(\0)$ is a fixed point of $\T$ so that $J_{\omega_1} = \T^{\xi}(\0)$, then it follows immediately that $J_{\omega_1}$ is lower semi-analytic, since $\sup_n f_n$ for a sequence of lower semi-analytic functions $f_n$ is a lower semi-analytic function~\cite[Lemma 7.30]{bs}. However, in general, $J_{\omega_1}$ is the pointwise supremum of an uncountable set of lower semi-analytic functions, so the preceding lemma does not follow easily from the argument we just mentioned. The line of analysis to establish the lemma is as follows.

We are going to apply a deep theorem from the descriptive set theory. This theorem concerns fixed points of a mapping $\Psi$ from the power-set of a Borel space $Y$ to the power-set of $Y$. More specifically, a fixed point of $\Psi$ is a set $E \subset Y$ satisfying $E = \Psi(E)$. 
For a given set $A \subset Y$, consider the set $A_{\omega_1}$ defined by the transfinite recursion,
\begin{equation} \label{eq-Aomg}
 A_0 =  A; \qquad A_\xi = \Psi \left( \bigcap_{\eta < \xi} A_\eta \right), \quad 0 < \xi < \omega_1; \qquad A_{\omega_1} = \bigcap_{\xi < \omega_1} A_\xi. 
\end{equation} 
The theorem states that if the mapping $\Psi$ is such that:
\begin{itemize}
\item[(i)] $\Psi$ is monotone (i.e., $\Psi(E_1) \subset \Psi(E_2)$ if $E_1 \subset E_2 \subset Y$), 
\item[(ii)] $\Psi(E) \subset E$ for $E \subset Y$, and $\Psi$ is \emph{uniformly analytic} in the sense that for any Polish space $\Omega$ and analytic set $C \subset \Omega \times Y$, the set $\{ (\omega, y) \in \Omega \times Y \mid y \in \Psi(C_{\omega}) \}$ is analytic (here $C_{\omega}=\{ y \in Y\mid (\omega, y) \in C \}$),
\end{itemize}
then for any analytic set $A \subset Y$, the set $A_{\omega_1}$ obtained by the recursion (\ref{eq-Aomg}) is analytic and satisfies $A_{\omega_1} = \Psi ( A_{\omega_1})$, and is indeed the largest subset of $A$ invariant under $\Psi$.

For a proof of the preceding theorem, see Zinsmeister \cite{Zin}, which also showed that the uniform analyticity condition in (ii) implies the monotonicity of $\Psi$ for analytic sets. According to \cite{MS92,Zin}, this theorem is a special case of a very general result of Moschovakis; see \cite{MS92,Zin} for related references.

To apply this theorem to our problem, we take $Y = S \times [ 0, +\infty]$, and we equate nonnegative functions with their epigraphs and equate the mapping $\T$ to a mapping that maps epigraphs to epigraphs. 
Specifically, for any function $J \geq 0$, we consider equivalently its epigraph $\epi(J)$ given by
$$ \epi(J) = \big\{ (x, z) \in S \times [0,+\infty] \ \big| \   J(x) \leq z \, \big\},$$
and we define a mapping $\Psi$ on the power-set of $S \times [0, +\infty]$ in such a way that for a lower semi-analytic function $J \geq 0$,
\begin{equation} \label{eq-psi}
    \epi\big(\T(J) \big) = \Psi \big(\epi(J)\big).
\end{equation}
In other words, $\Psi$ is the mapping equivalent to $\T$, on the space of epigraphs of functions. However, to apply the theorem, we also need to define $\Psi$ for all subsets in $S \times [0, +\infty]$, including the epigraphs of those non-measurable functions. This is achieved in \cite{MS92} by using an outer-integral formulation to define $\Psi$. Specifically, for any set $E \subset S \times [0, +\infty]$, we define
\begin{equation} \label{def-psi}
   \Psi(E) = E \, \bigcap \left\{ (x, z) \in S \times [0, + \infty] \ \Big | \  \inf_{u \in U(x)} \left\{ g(x, u) +  \big( q(\cdot \mid x, u) \times \lambda \big)^*\big(E^c\big) \,  \right\} \leq z \, \right\}.
\end{equation}   
Here $\lambda$ is the Lebesgue measure on $[0, +\infty]$; for $p \in \P(S)$, $(p \times \lambda)^*$ denotes the outer measure with respect to the product measure $p \times \lambda$ on $S \times [0, +\infty]$; and $E^c$ denotes the complement of a set $E$: $E^c = (S \times [0, +\infty] )\setminus E$.

By its definition (\ref{def-psi}), $\Psi$ is monotone and $\Psi(E) \subset E$. 
By \cite[Lemma 3.1]{MS92}, $\Psi$ is uniformly analytic.
\footnote{More specifically, let $\Omega$ be a Polish space and let $C$ be any analytic set in $\Omega \times Y$ where $Y=S \times [0,+\infty]$. Let $\Phi$ be the mapping which maps $E \subset Y$ to the set given by the expression in $\{\cdots\}$ in the definition~(\ref{def-psi}) for $\Psi$, so that $\Psi(E) = E \cap \Phi(E)$. \cite[Lemma 3.1]{MS92} shows that $\Phi$ is uniformly analytic, i.e., the set $D=\{(\omega, y) \mid \omega \in \Omega, y \in \Phi(C_\omega) \}$ is analytic. Then, $\{ (\omega, y) \mid \omega \in \Omega, y \in \Psi(C_\omega) \} = C \cap D$ is also analytic. Hence $\Psi$ is uniformly analytic.}
The mapping $\Psi$ also satisfies the desired equality~(\ref{eq-psi}): for a lower semi-analytic function $J \geq 0$, we have by Fubini's theorem,
$$ \big( q(\cdot \!\mid x, u) \times \lambda \big)^*\big(\epi(J)^c\big)  = \int_{S} J(x') \, q(dx' \!\mid x, u),$$
so
$$  \inf_{u \in U(x)} \left\{ g(x, u) +  \big( q(\cdot \mid x, u) \times \lambda \big)^*\big(\epi(J)^c\big) \right\} = T(J)(x)  $$
and $(x,z) \in \Psi\big(\epi(J)\big)$ if and only if
$$   J(x) \leq z \ \ \text{and} \ \  T(J)(x) \leq z   \quad \Longleftrightarrow \quad \max\big\{ J(x), \, T(J)(x) \big\} \leq z \quad \Longleftrightarrow \quad 
    (x, z) \in \epi\big(\T(J)\big).$$

Now we can apply the theorem from the descriptive set theory mentioned earlier. 
Let $A_{\omega_1}$ and $A_{\xi}$, $\xi < \omega_1$, be given by the transfinite recursion~(\ref{eq-Aomg}) with $A_0 = \epi(\0) = S \times [0,+\infty]$. 
Denote $J_\xi = \T^\xi(\0)$ for $\xi \leq \omega_1$; in particular $J_0 = \0$. Similar to \cite[Lemma 3.3]{MS92}, we show by transfinite induction that $A_{\omega_1} = \epi(J_{\omega_1})$.
We have $A_0=\epi(J_0)$. For any given $\xi < \omega_1$, suppose $A_\eta = \epi\big(J_\eta\big)$ for all $\eta < \xi$, and we have
\begin{align*}
  A_{\xi} & = \Psi \left( \bigcap_{\eta < \xi} A_\eta \right) = \Psi \left( \bigcap_{\eta < \xi} \epi\big(J_\eta\big) \right) \\
   & = \Psi \left( \epi\Big( \sup_{\eta < \xi} J_\eta \Big) \right) = \epi \left( \T \Big( \sup_{\eta < \xi} J_\eta \Big) \right) 
   = \epi \big( J_\xi \big),
\end{align*}
where we used Eq.~(\ref{eq-psi}) in the second to last equality, and we used the definition of $\T^\xi(\0)$ in the last equality [cf.\ Eq.~(\ref{eq-transvi0})].  
By transfinite induction~\cite[Theorem 1.3.1]{Dud02}, this implies
$A_{\xi} = \epi(J_\xi)$ for all $\xi < \omega_1,$
and consequently,
$$ A_{\omega_1} = \bigcap_{\xi < \omega_1} A_{\xi} =  \bigcap_{\xi < \omega_1} \epi(J_\xi) = \epi\Big(\sup_{\xi < \omega_1} J_\xi \Big) =\epi\big( J_{\omega_1} \big),$$
where in the last equality we used the definition $\T^{\omega_1}(\0) = \sup_{\xi < \omega_1} J_\xi$ [cf.~Eq.~(\ref{eq-transvi1})].

Now by the theorem mentioned earlier,
$A_{\omega_1}$ is analytic and $A_{\omega_1} = \Psi(A_{\omega_1})$. 
This means, first, that $\epi(J_{\omega_1})=A_{\omega_1}$ is an analytic set, which is equivalent to that $J_{\omega_1}$ is lower semi-analytic (cf.\ \cite[p.~186]{bs}). 
Second, 
$$ \epi \big( J_{\omega_1} \big) = \Psi\big(\epi(J_{\omega_1})\big) = \epi\big(\T(J_{\omega_1})\big),$$
where the second equality follows from Eq.~(\ref{eq-psi}). Since two functions are identical if and only if they have the same epigraph, from the preceding relation we obtain
$$ J_{\omega_1} = \T(J_{\omega_1}) = \max \big\{ J_{\omega_1}, T (J_{\omega_1} ) \big\} \geq T(J_{\omega_1}).$$
This establishes~\cref{lma-rel2P}. 

\end{appendices}

\end{document}